\documentclass[11pt]{article}
\usepackage{pstricks}
\usepackage{graphicx}

\usepackage{amsmath,amstext,amssymb,amsfonts}
\usepackage{amsthm}

\newtheorem{lemma}{Lemma}[section]
\newtheorem{theorem}[lemma]{Theorem}
\newtheorem{corollary}[lemma]{Corollary}
\newtheorem{definition}[lemma]{Definition}

\newtheorem{example}[lemma]{Example}

\newtheorem{remark}[lemma]{Remark}

\def\S{\mathbb{S}}

\def\X{\mathbb{X}}

\def\G{\mathbb{G}}
\def\A{\mathbb{A}}
\def\H{\mathbb{H}}
\def\R{\mathbb{R}}

\def\B{\mathcal{B}}

\def\dF{\mathbb{F}}
\def\dM{\mathbb{M}}

\begin{document}




\title{Sufficiency of Deterministic Policies for Atomless Discounted and Uniformly Absorbing MDPs with Multiple Criteria 
\thanks{The research of the
first author was partially supported by the National Science Foundation [Grant CMMI-1636193].}} 


\author{Eugene~A.~Feinberg\thanks{Department of Applied Mathematics and
Statistics,
 Stony Brook University,
Stony Brook, NY 11794-3600, USA, eugene.feinberg@sunysb.edu}
\and Alexey~B.~Piunovskiy\thanks{Department of Mathematical Studies,
University of Liverpool, Liverpool, L69 7ZL, UK, piunov@liverpool.ac.uk}
}
\maketitle

%
%

%
%
%

\begin{abstract}
This paper studies Markov Decision Processes (MDPs) with atomless initial state distributions and atomless transition probabilities.  Such MDPs are called atomless.  The initial state distribution is considered to be fixed.  We show that for discounted MDPs with bounded one-step reward vector-functions, for each policy there exists a deterministic (that is, nonrandomized and stationary) policy with the same performance vector.  This fact is proved in the paper for a more general class of uniformly absorbing MDPs with expected total costs, and then it is extended under certain assumptions to MDPs with unbounded rewards. For problems with multiple criteria and constraints, the results of this paper imply that for atomless MDPs studied in this paper it is sufficient to consider only deterministic policies, while  without the atomless assumption it is well-known that randomized policies can outperform deterministic ones.  We also provide an example of an MDP demonstrating that, if a vector measure is defined on a standard Borel space, then Lyapunov's convexity theorem is a special case of the described results.

 \end{abstract}












\section{Introduction}\label{s1}
This paper studies Markov Decision Processes (MDPs) with multiple criteria when each criterion is evaluated by the expected total discounted rewards or costs.  The paper also studies more general  uniformly absorbing MDPs.  The number of criteria is finite, and the initial state distribution is fixed.  For each criterion there is a function of one-step rewards, and the performance of each policy is evaluated by the finite-dimensional vector, whose coordinates are expected total rewards for the corresponding reward functions.  For each policy this vector is called a performance vector.    An MDP is called atomless, if the initial state distribution and transition probabilities are atomless.  In general, constrained optimization requires the use of randomized decisions.  However, for atomless problems nonrandomized policies are optimal under broad conditions.

The first results of this kind were established by Dvoretzky et al.~\cite{DWW,DWW1}, who proved that for a one-step problem with multiple  atomless initial distributions,  multiple reward functions and finite action sets, the expected reward-vector achieved by an arbitrary policy can be achieved by a nonrandomized policy.  The case of multiple initial distributions can be reduced to a single initial distribution by using the Radon-Nikodym theorem; see~\cite{FP} or Example~\ref{ex10.2}. So, the mentioned result from Dvoretzky et al.~\cite{DWW,DWW1} can be interpreted as a fact for one-step atomless MDPs. As was observed by Feinberg and Piunovskiy~\cite{FP}, this result holds for infinite action sets; see also Ja\'skiewicz and Nowak \cite{JN} for the generalization to conditional expectations. The proof in  Dvoretzky et al.~\cite{DWW,DWW1} is based on Lyapunov's convexity theorem, that states that the range of a finite atomless vector-measure is a convex  compact subset of the Euclidean space.

Feinberg and  Piunovskiy~\cite{b6,b5} proved that for atomless MDPs with a given initial state distribution and with multiple expected total rewards, for every policy there is a nonrandomized Markov policy with the same performance vector. In \cite{b6} this fact was proved for MDPs with weakly continuous transition probabilities and with weakly continuous reward functions.  The proof in \cite{b6} is based on geometric arguments.  In \cite{b5} this fact is proved for arbitrary atomless MDPs with expected total rewards, and the proof is based on Lyapunov's convexity theorem.

In this paper we prove that for an atomless discounted MDP with multiple criteria and bounded reward functions, for each policy there exists a deterministic (that is, nonrandomized and stationary) policy with the same performance vector.  In fact, we prove this result for uniformly absorbing MDPs with the expected total rewards.  This is a more general class of MDPs than discounted ones.  The proof for deterministic policies is much more difficult than the proofs for nonrandomized Markov ones provided in  \cite{b6} and \cite{b5}.  In addition, the proofs in this paper use and extend  geometric methods introduced in \cite{b6} instead of applying Lyapunov's convexity theorem.   Example~\ref{ex10.2} demonstrates that Lyapunov's convexity theorem can be interpreted as a one-step version of the main result of this paper.

For discounted MDPs with multiple criteria and constraints, under certain conditions there exist (randomized) stationary optimal policies; see Altman~\cite{Al}, Feinberg and Shwartz~\cite{FS},  Hern\'andez-Lerma, and Gonz\'alez-Hern\'andez~\cite{HLGH}, Piunovskiy~\cite{b7}.  The results of this paper imply the existence of optimal deterministic policies for constrained atomless discounted MDPs and for constrained atomless uniformly absorbing MDPs if  optimal policies exist.

 The main result of this paper, Theorem~\ref{tgmain}, states that the sets of performance vectors for all policies and for deterministic policies coincide.  In order to prove the main result, we deal with three types of subsets of linear spaces: the set of strategic measures, the set of occupancy measures, and the set of performance vectors.   For a given policy, the strategic measure is the probability distribution of all state-action trajectories, and the occupancy measure is the measure on the product of the state and action spaces, and the value of this measure on each measurable set is the expected total number of times when the corresponding actions are selected at the corresponding states. The set of performance vectors  (strategic measures, occupancy measures)  consists of performance vectors (strategic measures, occupancy measures) for all policies. The set of performance vectors is a projection of the set of occupancy measures, and the set of occupancy measures is a projection of the set of strategic measures.  Projections inherit certain properties of the sets from which they are projected.  These properties include convexity and compactness.

The set of all strategic measures is convex; Dynkin and Yushkevich~\cite[Section]{DY}.   Therefore,  the set of all occupancy vectors and the set of all performance vectors are convex. Under certain conditions the sets of strategic measures is compact. Sch\"al~\cite{b8} introduced two such conditions: (S) and (W).  Condition (S) assumes setwise continuity of transition probabilities, and Condition (W) assumes weak continuity of transition probabilities.  In the both cases, appropriate continuity properties are assumed for reward functions.  
In particular, condition (S) holds for MDPs with finite action sets. Under the mentioned conditions, compactness properties also hold for the sets of all occupancy measures and all performance vectors.

For discounted and absorbing MDPs, if the initial distribution is fixed, then for each policy there exists a stationary policy with the same occupancy measure; see \cite{Al, Bor, FR, FSo, b7, Pi98}.  Therefore, the sets of all occupancy measures and all performance vectors coincide with the corresponding sets for all stationary policies.  The nontrivial step in proving Theorem~\ref{tgmain} is to show that the sets of performance vectors for all stationary and for all deterministic policies coincide.

The important and nontrivial step is to prove that for an atomless MDP the set of performance vectors for all deterministic policies is convex.  This fact is nontrivial even for the case of one criterion.  Example~\ref{ex10.2} demonstrates that for multiple criteria this fact is a nontrivial extension of Lyapunov's convexity theorem  for a standard Borel space. In order to prove this fact, we show that the set of occupancy measures endowed with the topology of setwise convergence is path-connected.  Therefore, being its projection,  the set of performance vectors is a connected subset of the Euclidean space.  Thus, for the single-criterion case, this set is a connected subset of a line. Therefore, it is convex.  The case of multiple criteria is studied by induction using the dimensionality reduction technique introduced in this paper.

 Section~\ref{s2} of this paper introduces the basic definitions for the discounted case and formulates the main result for discounted MDPs.  Section~\ref{s3} describes absorbing and uniformly absorbing MDPs, formulates the main result for uniformly absorbing MDPs, and shows that a discounted MDP is a particular case of a unformly absorbing MDP.  
 Section~\ref{s4} studies the properties of occupancy measures.  Section~\ref{s5} describes Condition~(S), which is sufficient for compactness of the sets of all strategic measures, all occupancy measures, and all performance vectors.  In particular, this condition holds for an MDP with finite action sets.  Section~\ref{s6} describes submodels and dimensionality reduction.  Section~\ref{s7} introduces an MDP generated by two deterministic policies and describes continuity properties for such MDPs.  Section~\ref{s8} establishes path-connectedness of the sets of occupancy measures for all deterministic policies for atomless MDPs.  This property implies that the set of all performance vectors for deterministic policies is path-connected. Thus, for a single-criterion problem, this set is convex.  The proof of the main theorem is provided in Section~\ref{s9}.  Section~\ref{s10} provides the results for unbounded reward vector-functions by using the standard weighted norm approach.  These results are used in Section~\ref{s11} to show that for standard Borel spaces Lyapunov's convexity theorem is a special case of the results of this paper.  
\section{Main result for Discounted MDPs}\label{s2}
We start with some definitions.  Recall that two measurable spaces $(E,{\cal E})$ and $(D,{\cal D})$ are called isomorphic, if there exists a one-to-one measurable correspondence $f$ between them such that the correspondence $f^{-1}$ is measurable. A Polish space is a complete separable metrizable space.  A standard Borel space is a measurable space isomorphic to a Borel subset of a Polish space.  Properties of standard Borel spaces can be found in Bertsekas and Shreve~\cite{b3}, Dynkin and Yushkevich~\cite{DY}, Kechris~\cite{Ke}, and  Srivastava~\cite{Sr}.  In particular, a standard Borel space is either finite or countable, or it has the cardinality of the continuum.  Two standard Borel spaces with the same cardinality are isomorphic. We always consider Borel $\sigma$-fields on topological and metric spaces.  In particular, a standard Borel space with the cardinality of continuum is isomorphic to the interval $[0,1].$  For two measurable spaces $(E,{\cal E})$ and $(D,{\cal D}),$ a transition probability $q$ defines a probability measure $q(\cdot|d)$ on $(E,{\cal E})$ for each $d\in D$ such that $q(C|\cdot)$ is a measurable function on $(D,{\cal D})$ for each $C\in\cal E.$ We recall that a measure  $\nu$ on a standard Borel $(D,{\cal D})$ space is called atomless if $\nu(d)=0$ for all $d\in D;$ here and below we omit curly brackets in the expressions like $\nu(\{x\})$ and $p(\{y\}|x,a).$ 

A discounted MDP is defined by the following objects:
\begin{itemize}
\item[(i)]
a standard Borel state space $(\X,{\cal X}),$
\item[(ii)] a standard Borel action space $(\A,{\cal A}),$
\item[(iii)] nonempty sets of  actions $A(x)\in\cal A$ available at states $x\in \X,$ such that 
${\rm Gr}_\X(A):=\{(x,a)\in \X\times \A:\, x\in \X,\ a\in A(x)  \}
$
is a measurable subset of $(\X\times \A,\cal{X}\otimes\cal{A}),$
\item[(iv)] a transition probability $p$ from $\X\times \A$ to $\X,$
\item[(v)] an initial state distribution $\mu,$ which is a probability measure on $(\X,{\cal X}),$
\item[(vi)] a bounded measurable reward vector-function $r:\X\times\A\mapsto \R^N,$ where $N$ is a natural number,
\item[(vii)]   a discount factor $\beta\in [0,1).$
 \end{itemize}
 \begin{definition}
An MDP is called atomless if $\mu(x)=0$ and  $p(y|x,a)=0$ for all $x,y\in \X$ and  $a\in A(x).$
\end{definition}
If an action $a\in A(x)$ is chosen at a state $x\in \X,$ then the process moves to the next state according to the probability distribution $p(\cdot|x,a)$ and the vector reward $r(x,a)=(r^{(1)}(x,a), r^{(2)}(x,a),\ldots,r^{(N)}(x,a))$  is collected according to criteria $1,2,\ldots,N.$  To avoid a trivial situation, when a policy cannot be defined, we always assume that there exists a measurable mapping $\phi:\X\mapsto A$ such that $\phi(x)\in A(x)$ for all $x\in\X.$  Such mapping is called a selector.

Consider the sets of possible finite histories $\H_t:=\X\times (\A\times \X)^t$ up to time $t=0,1,\ldots\ .$  A policy $\pi$ is a sequence of transition probabilities $\pi_t,$ $t=0,1,\ldots,$ from $\H_t$ to $A$ such that $\pi(A(x_t)|h_t)=1$ for each $h_t=(x_0,a_0,x_1,\ldots,x_t)\in \H_t.$  A policy is called nonrandomized if each transition probability $\pi_t(\cdot|h_t),$ $t=0,1,\ldots,$  is concentrated at one point.  A policy $\pi$ is called Markov, if for each $t=1,2,\ldots$ the values of probabilities $\pi_t(\cdot|x_0,a_0,\ldots,x_t)$ are the functions of $x_t.$  A Markov policy is called stationary if $\pi_t(\cdot|x)=\pi_s(\cdot|x)$ for all $x\in\X$ and for all $s,t=0,1,\ldots\ .$ A transition probability $\pi_t$ for a stationary policy $\pi$ is also denoted as $\pi.$ A nonrandomized Markov policy is defined by a sequence of selectors $\{\phi_t\}_{t=0,1,\ldots}.$  These selectors are equal for a nonrandomized stationary policy.  A nonrandomized stationary policy $\phi$ is called deterministic, and we identify it with the selector $\phi.$ We denote by $\Pi,$ $\dM,$ $\S,$ and $\dF$ the sets of all,  nonrandomized Markov, stationary, and deterministic policies respectively.  Observe that $\dF\subset\dM\subset\Pi$ and $\dF\subset\S\subset\Pi.$

The existence of the selector means that  $\dF\ne\emptyset.$  This assumption does not limit the generality of the results of this paper.
If $\dF=\emptyset,$ then $\Pi=\emptyset;$ see Dynkin and Yushkevich~\cite[Sections 3.1 and 3.2]{DY}. Therefore, if $\dF=\emptyset,$ then the main result of the paper, Theorem~\ref{tgmain}, is equivalent to the trivial identity $\emptyset=\emptyset.$

The two special features of the introduced model are: (i) the rewards are vector-valued, and (ii) the initial distribution $\mu$ is fixed.  However, we  consider additional initial distributions and initial states in auxiliary results in a few places in this paper.  Whenever we consider other initial distributions rather than $\mu,$ we specify them in notations.

According to the Ionescu Tulcea theorem, an initial probability distribution $\mu$ on the state space $\X$ and transition probabilities $\pi_t$ and $p$ define a unique probability measure $P^\pi$ on the  countable product $\H_\infty:=\X\times (\A\times \X)^\infty$ endowed with the $\sigma$-field $\cal{X}\otimes (\cal{A}\otimes \cal{X})^\infty.$  Expectations with respect to this probability is denoted by $E^\pi.$

\begin{remark}\label{rem1}
{\rm The corresponding probabilities and expectations are defined for each initial probability distribution $\nu$ on $(\X,{\cal X}).$  In this case, they are denoted as $P_\nu^\pi$ and $E_\nu^\pi.$  That is, $P^\pi:=P_\mu^\pi$ and $E^\pi:=E_\mu^\pi.$  If a probability measure $\nu$ is concentrated at a point $x\in\X,$ that is, $\nu(x)=1,$ we shall write $P_x^\pi$ and $E_x^\pi$ instead of $P_\nu^\pi$ and $E_\nu^\pi$ respectively.}
\end{remark}

For an initial state distribution $\mu$ and a policy $\pi,$ the vector of expected total discounted rewards is
\[v^\pi_\beta:=E^\pi \sum_{t=0}^\infty \beta^t r(x_t,a_t).\]
For a set of policies $\Delta\subset\Pi,$ the set of all performance vectors is ${\cal V}_\beta^\Delta:=\{v_\beta^\pi :\, \pi\in\Delta\}.$

Denote ${\cal V}_\beta:= {\cal V}_\beta^\Pi. $ It is obvious that ${\cal V}_\beta^\dF\subset {\cal V}_\beta\subset\R^N$ and, in general, it is possible that  ${\cal V}_\beta^\dF\ne {\cal V}_\beta.$  For example, if $\X$ and $\A$ are finite sets, then the set ${\cal V}_\beta^\dF$ is finite while the set ${\cal V}_\beta$ may have the cardinality of the continuum.  In fact, for problems with finite state and action sets, ${\cal V}_\beta$ is a convex hull of  ${\cal V}_\beta^\dF;$ see e.g. Feinberg and Rothblum~\cite[Theorem 6.1]{FR}. 
According to the following theorem, which  is the main result of this paper for discounted MDPs, the situation is different for atomless MDPs.
\begin{theorem}\label{tmain}
 For an atomless MDP  ${\cal V}_\beta^\dF= {\cal V}_\beta.$
\end{theorem}
In Section~\ref{s3} we formulate a more general result, which is proved later in this paper.  

\section{Absorbing MDPs and the Main Result}\label{s3}
We start this section with the definition of  the expected total reward under fairly general condition and for the case of a single criterion, that is,  $N=1.$ In this case,  $r$ is a bounded real-valued function, but in formula \eqref{defdiscpos} and in Definition~\ref{dEF3.1} we do not assume that $r$ is bounded. Then we define absorbing and uniformly absorbing MDPs,  formulate the main result of this paper, Theorem~\ref{tgmain}, and show that it is more general  than Theorem~\ref{tmain}, which states the sufficiency of deterministic policies for atomless discounted MDPs.

We recall that the initial state distribution $\mu$ is fixed.  
For an arbitrary nonnegative measurable function $r,$ the expected total reward for a policy $\pi$ is
\begin{equation}\label{defdiscpos} v^\pi:=E^\pi\sum_{t=0}^\infty r(x_t,a_t)=\lim_{n\to\infty}E^\pi\sum_{t=0}^{n-1} r(x_t,a_t) ,\end{equation}
where the second equality follows from the monotone convergence theorem.

  For a number $c,$ let us denote $c^+:=\max\{c,0\}$ and $c^-:=-\min\{c,0\}.$  For a policy $\pi\in\Pi,$  we consider positive values $v_+^\pi$ and  $v_-^\pi$ defined by \eqref{defdiscpos} with the rewards $r(x,a)$  substituted with the rewards $r^+(x,a)$ and $r^-(x,a)$ respectively.
\begin{definition}\label{dEF3.1}
If $\min\{v^\pi_+,v^\pi_-\}<+\infty,$ then the expected total reward $v^\pi$ is well-defined and $v^\pi:=v^\pi_+-v^\pi_-.$
\end{definition}

If $v^\pi$ is well-defined, then the  equalities in \eqref{defdiscpos} hold because they hold for rewards $r^+$ and $r^-$ and at least one of the numbers  $v_+^\pi$ and $v_-^\pi$ is finite.

Now let $N>1.$  Then $v_+^\pi$ and    $v_-^\pi$ are defined as $N$-dimensional vectors of the expected total rewards whose coordinates are the expected total rewards for positive and negative parts of the corresponding coordinates of the vector-function $r.$ The vector $v^\pi$ is well-defined if so is each of its $N$ coordinates. In this case, as explained above, $v^\pi:=v^\pi_+-v^\pi_-,$  and the second equality in \eqref{defdiscpos} holds.

\begin{remark}
{\rm For an initial probability distribution $\nu$  on $(\X,\cal{X}),$ that can be different from $\mu,$ we shall use the notations $v(\nu),$ $v_+(\nu),$ and $v_-(\nu)$ respectively.  With a small abuse of notations, we shall write $v(x),$ $v_+(x),$ and $v_-(x)$ respectively, if the probability measure $\nu$ is concentrated at the point ${x}\in\X.$}
\end{remark}

Now we introduce an absorbing MDP.
Let the standard Borel state space of this MDP be denoted by $\bar \X.$  We use the same notations and assumptions for the standard Borel action space $\A,$ sets of available actions $A(\cdot),$ transition probability $p, $ initial state distribution $\mu,$ and reward vector $r$ as in the previous section.

Let $T^x$ denote the first time a stochastic sequence $h=x_0,x_1,\ldots$ with values in $\bar \X$  reaches the state  $x\in\bar{\mathbb X};$
$T^x(h):=\inf\{t= 0,1,\ldots:\, x_t=x\}.$
\begin{definition}\label{defabs}
For the initial probability distribution $\mu,$ an MDP is called absorbing, if there exists a state ${\bar x}\in{\bar \X}$ with the following properties:

 (i) $\mu({\bar x})=0;$

 (ii)  $A({\bar x})=\{{\bar a}\}$ for some ${\bar a}\in \A,$
  $p({\bar x}|{\bar x},{\bar a})=1,$ and $r^{(i)}({\bar x},{\bar a})=0$ for all $i=1,\ldots,N;$

 (iii) there exists a finite constant $L$ such that,  for all policies $\pi\in\Pi,$
 \begin{equation}\label{eqabsiii}E^\pi T^{\bar x}\le L.\end{equation}
\end{definition}
\begin{remark}\label{rem2.0}
The state $\bar{x}$ is fictitious in the sense that  under every policy this state is absorbing, there is no choice of decisions  at $\bar{x},$  and all the rewards are equal to 0 at this state.  After the system hits state $\bar{x},$ it is impossible to control it.  Therefore, the set $\bar{\X}\setminus\{\bar{x}\}$ plays the same role for absorbing MDPs as the state space $\X$ for discounted MDPs; see the notation in formula~\eqref{eqdefdessp341}.
\end{remark}

\begin{remark}\label{rem2}
{\rm We make assumption (i) in Definition~\ref{defabs} for convenience only.  All the results in this paper hold without this assumption.  In principle, it is possible to consider other initial distributions than $\mu.$  If an MDP is absorbing for an initial distribution $\nu,$ which may differ from $\mu,$  then this is stated explicitly in this paper. Of course, the value of the upper bound $L$ may depend on the initial distribution.  In some publications, including~\cite{Al,FR}, absorbing measurable sets are considered instead of absorbing states.  These formulations are equivalent because the states in an absorbing set can be merged into a single state.}
\end{remark}

  Observe that $T^{\bar x}=\sum_{t=0}^\infty I\{t<T^{\bar x}\},$ where $I$ is the indicator function.  We recall that assumption (iii) in Definition~\ref{defabs} is equivalent to the validity of \eqref{eqabsiii} for all deterministic policies $\phi\in\dF$ instead of arbitrary policies $\pi\in\Pi;$ see Feinberg and Rothblum~\cite[p. 132]{FR}. If we interpret $T^{\bar x}$ as the time, when the process stops, then  \eqref{eqabsiii} means that the average life-time of the process is uniformly bounded for all policies given the initial state distribution $\mu.$    For an absorbing MDP, we fix an arbitrary state $\bar x$ described in Definition~\ref{defabs} and  set \begin{equation}\label{eqdefdessp341}\X:={\bar \X}\setminus \{{\bar x}\}.\end{equation}

Let us consider an absorbing MDP.  Recall that the reward vector-function $r$  is bounded  and  $r(\bar x,\bar a)=0$. In view of Definition~\ref{defabs}(ii, iii), the expected total rewards $v^\pi$ are well-defined for all policies $\pi$ and
\begin{equation}v^\pi=\lim_{n\to\infty}E^\pi \sum_{t=0}^{n-1} r(x_t,a_t) = E^\pi\sum_{t=0}^{\infty} r(x_t,a_t) =E^\pi\sum_{t=0}^{\infty} r(x_t,a_t)I\{x_t\in \X\}=E^\pi\sum_{t=0}^{T^{\bar x}-1} r(x_t,a_t),\label{defvpigen}\end{equation}
where  the first two equalities follow from \eqref{defdiscpos} and the last two ones follow from Definition~\ref{defabs}(ii). 
For $\Delta\subset \Pi,$  the sets of performance vectors generated by policies from $\Delta$ is
${\cal V}^\Delta:=\{v^\pi :\, \pi\in\Delta\}.$  We also use the notation \[{\cal V}:={\cal V}^\Pi.\]

For  an absorbing MDP, the monotone convergence theorem implies that for every policy $\pi$ 
\[\lim_{n\to\infty}E^\pi\sum_{t=n}^{\infty} I\{t<T^{\bar x}\}=0.\]
Definition~\ref{defabsu} states the stronger equality. Recall that $\dM$ is the set of all nonrandomized Markov policies and the initial measure $\mu$ is fixed.
\begin{definition}\label{defabsu}
An absorbing MDP is called uniformly absorbing if
\begin{equation}\label{elimsup} \lim_{n\to\infty}\sup_{\pi\in \dM} E^\pi\sum_{t=n}^\infty I\{t<T^{\bar x}\}= 0.
\end{equation}
\end{definition}
 Example~\ref{ex3.12} describes an absorbing MDP, which is not uniformly absorbing. We remark that  the supremum in  \eqref{elimsup} is equal to the same supremum over the set of all policies $\pi\in \Pi;$ Feinberg~\cite[Theorem 3]{F82}. Recall that $E^\pi I\{t<T^{\bar x}\}=P^\pi \{T^{\bar x}>t\}$ and $E^\pi T^{\bar x}=\sum_{t=0}^\infty P^\pi \{T^{\bar x}>t\}$. Since $E^\pi\sum_{t=n}^\infty I\{t<T^{\bar x}\}= E^\pi T^{\bar x}- E^\pi\sum_{t=0}^{n-1} I\{t<T^{\bar x}\},$  assumption~\eqref{elimsup} means that the MDP is absorbing and   the convergence $E^\pi\sum_{t=0}^{n-1} I\{t<T^{\bar x}\}\uparrow E^\pi T^{\bar x}$ as $n\to\infty$  takes place uniformly in $\pi\in\Pi.$  Since the vector-function $r$ is bounded, the convergence in \eqref{defdiscpos} is  uniform in $\pi\in\Pi$ for a uniformly absorbing MDP.

\begin{definition}\label{def2.4}
An absorbing  MDP is called atomless if $\mu(x)=0$ and $p(y|x,a)=0$ for all $x,y\in\X$ and $a\in A(x).$
\end{definition}
 In some sense, Definition~\ref{def2.4} means that the state $\bar x$ is considered to be outside of the state space.  Of course, a uniformly absorbing MDPs is absorbing, and Definition~\ref{def2.4} applies to uniformly absorbing MDPs too.

As explained later in this section, the following theorem, which is the main result of this paper, generalizes Theorem~\ref{tmain} that states the similar statement for discounted MDPs.
\begin{theorem}\label{tgmain} For a uniformly absorbing atomless MDP,
${\cal V}^\dF =  {\cal V}.$
\end{theorem}
The following corollary is an equivalent formulation of Theorem~\ref{tgmain}.
\begin{corollary}\label{cthmain} For a uniformly absorbing atomless MDP, for every policy $\pi\in\Pi$ there exists a deterministic policy $\phi$ such that $v^\phi=v^\pi.$
\end{corollary}
For total-reward MDPs, the performance set $\cal V$ is convex.  This simple fact follows from the convexity of the set of strategic measures; see Dynkin and Yushkevich~\cite[Section 5.5]{DY} or, for absorbing MDPs,  see Lemma~\ref{lem3.1} below.  This fact and Theorem~\ref{tgmain} imply the following corollary.
\begin{corollary}\label{convex}
For a uniformly absorbing atomless MDP, the set ${\cal V}^\dF$ is convex.
\end{corollary}
Let us show that Theorem~\ref{tgmain} is more general than Theorem~\ref{tmain}. Recall that, if an initial probability distribution $\nu$ is concentrated at one state $x\in\X,$ then, according to Remark~\ref{rem1}, we usually write $E_x^\pi$ instead of $E_\nu^\pi.$  The following lemma provides a natural sufficient condition under which an absorbing MDP is  uniformly absorbing.
\begin{lemma}\label{lbound} Consider an MDP with a standard Borel state space $\bar \X$ and with a state $\bar x\in\bar \X$ such that $A(\bar x)$ is a singleton and $p({\bar x}|{\bar x}, {\bar a})=1,$ $r({\bar x}, {\bar a})=0,$  where $A(\bar{x})=\{\bar a\}.$ If there is a finite constant $L$ such that $E_x^\phi T^{\bar x}<L$ for all $x\in \X={\bar{\X}}\setminus\{\bar{x}\}$ and for all $\phi\in\dF,$  then this MDP is uniformly absorbing for all initial state distributions $\mu$ on $\X.$
\end{lemma}
\begin{proof}  Let us fix an arbitrary initial probability distribution $\mu$ on $ \X.$ As is mentioned after Definition~\ref{defabs}, $\sup_{\pi\in\Pi} E_x^\pi T^{\bar x}=\sup_{\phi\in\dF} E_x^\phi T^{\bar x}$ for all $x\in\X.$ Therefore, $E_x^\pi T^{\bar x}\le L$ for all $x\in \X$ and  for all $\pi\in\Pi.$
 This implies that $E^\pi T^{\bar x}\le L$ for all $\pi\in\Pi.$
In view of Markov's inequality, for an arbitrary policy $\pi\in\Pi$ and for  $n=0,1,\ldots,$
\begin{equation}\label{eqMar}
P^\pi \{T^{\bar x}> n\}\le (n+1)^{-1}E^\pi T^{\bar x}\le (n+1)^{-1}L.
\end{equation}
For an arbitrary nonrandomized Markov policy $\phi=(\phi_0,\phi_1,\ldots)$ and for $n=0,1,\ldots,$ let us define by $\phi^{+n}$ the shifted nonrandomized Markov policy $\phi^{+n}=(\phi^n, \phi^{n+1},\ldots ).$  Then
\[ E^\phi\sum_{t=n}^\infty I\{t<T^{\bar x}\}=E^\phi E_{x_n}^{\phi^{+n}} \sum_{t=0}^\infty I\{t<T^{\bar x}\}=E^\phi E_{x_n}^{\phi^{+n}}T^{\bar x}\\ \le   E^\phi I\{n<T^{\bar x}\}L=  LP^\phi\{T^{\bar x}>n\}\le (n+1)^{-1}L^2,\]
which implies \eqref{elimsup}, where the first inequality follows from $\{x_n\in \X\} = \{n<T^{\bar x}\}$ $P^\pi$-{\it a.s.} and $E_x^\pi T^{\bar x}\le L$ for all $\pi\in\Pi$ and all $x\in \X,$ and the last inequality follows from \eqref{eqMar}. 
\end{proof}
\begin{lemma}\label{lreduction}
Theorem~\ref{tgmain} implies Theorem~\ref{tmain}.
\end{lemma}
\begin{proof}
Consider a discounted MDP.  The following  transformation into an absorbing MDP is well-known; see e.g., Altman~\cite[p. 137]{Al}. Let us add an additional point $\bar x$ to the state space $\X$ and consider the new transition probability $\bar p$ defined by
\begin{equation*} {\bar p}(Y|x,a):= 
\begin{cases}
\beta p(Y|x,a),
&\mbox{if $x\in \X,$ $Y\in\cal X$},\\
1-\beta, 
&\mbox{if\ $x\in \X,$ $Y=\{\bar x\}$},\\
1, &\mbox{if\ $x={\bar x}\in Y.$}
\end{cases}
\end{equation*}
Then ${\cal V}^\dF = {\cal V}_\beta^\dF$ and $ {\cal V}={\cal V}_\beta.$ The new MDPs is absorbing.  It is atomless if and only if the original discounted MDP is atomless. Since $E_x^\pi T^{\bar x}=(1-\beta)^{-1},$ Lemma~\ref{lbound} implies that the new model is uniformly absorbing.
\end{proof}
Of course, the transformation of a discounted MDP into an absorbing one is trivial.  However, under certain conditions it is also possible to transform an absorbing MDP into a discounted one; see Feinberg and Huang~\cite{FH, FH1}.

The following example describes  an absorbing MDP, which is not uniformly absorbing.
\begin{example}\label{ex3.12}
Let $\X:=\{(i,j): i=0,1,\dots,\ j=0,1,\ldots,2^i-1\}$, ${\bar x}:=0,$ A=\{c,s\}, where $c$ stands for ``continue" and $s$ stands for ``stop", and
\begin{equation*}A(x):=\begin{cases} \{c,s\},  &{\rm if\ }  x=(i,0),\ i=0,1,\ldots,\\
\{s\}   &{\rm otherwise,}
\end{cases}
\end{equation*}
and for $i=0,1,\ldots$
\begin{equation*}
p(y|x,a)= \begin{cases} 0.5, &{\rm if\ } a=c,\ x=(i,0),\  y= 0 {\ \rm or\ } y=(i+1,0),\\
1, &{\rm if\ } a=s\ {\rm and\ either}\ x=(i,j),\ j=0,\ldots, 2^i-2, \ y=(i,j+1)\ {\rm or}\ x=(i,2^i-1),\ \\ & y=0.\end{cases}
\end{equation*}
In addition $\mu (0,0)=1.$  In this example, the process starts at the state $(0,0).$ At each state $(i,0),$ $i=0,1\ldots,$ the decision maker can either continue or stop the process.  If the process is continued at state $(i,0),$ then it  moves  with probabilities 0.5 either to state $(i+1,0)$  or to state $\bar x.$  If the process is stopped at state $(i,0)$, then it makes $2^i$ additional deterministic moves until it hits the absorbing state $\bar x=0$ and stops.
Let $\phi^\infty$ be the deterministic policy that always chooses an action $c$ at the states $(i,0),$ $i=0,1,\ldots\ .$  Under this policy, $T^{\bar x}$ has the geometric distribution with the success probability 0.5 at each step. Therefore, $E^{\phi^\infty} T^{\bar x} =2.$
 Now let $\phi^n$ be a deterministic policy  choosing the action $s$ at the state $(n,0)$ and the action $c$ at the states $(i,0)$ with $i=0,1,\ldots,n-1,$ where  $n=0,1,\ldots\ .$
Then $E^{\phi^n} T^{\bar x}=E^{\phi^n}[\sum_{t=0}^{n-1} I\{t<T^{\bar x}\}+2^n I\{t<T^{\bar x}\}]= \sum_{t=0}^{n-1} 2^{-t}+2^n2^{-n}=3-2^{-n+1}.$  Thus, $E^\phi T^{\bar x}\le 3$ for all $\phi\in\dF.$ So, this MDP is absorbing.
However, $ \lim_{n\to\infty}\sup_{\pi\in \dM} E^\pi\sum_{t=n}^\infty I\{t<T^{\bar x}\}\ge\lim_{n\to\infty}E^{\phi^n}\sum_{t=n}^\infty I\{t<T^{\bar x}\}= \lim_{n\to\infty}2^{-n}2^n=1.$   Thus, this MDP is not uniformly absorbing. 
%
\end{example}

\section{Occupancy Measures and their Properties}\label{s4}
For an absorbing MDP, a policy $\pi,$ and an initial state distribution $\mu$ on $\X,$ the finite occupancy measure $Q^\pi(\cdot)$ on $\X\times \A$ is defined by
\[Q^\pi(Y\times B):= E^\pi \sum_{t=0}^{T^{\bar x}-1} I\{x_t\in Y, a_t\in B\}=\sum_{t=0}^\infty P^\pi\{x_t\in Y, a_t\in B\},\qquad Y\in {\cal X},\ B\in {\cal A}.
\]
Let $q^\pi( Y):=Q^\pi(Y\times \A),$ where $Y\in{\cal X}.$  Observe that $q^\pi( \X)=E^\pi T^{\bar x}\le L.$ In addition,
\begin{equation}\label{eqingjf} v^\pi=\int_\X\int_A r(x,a)Q^\pi(dxda).
\end{equation}
 The set of occupancy measures for the initial distribution $\mu$ and for all policies from $\Delta\subset\Pi$ is
\[{\cal M}^\Delta :=\{Q^\pi(\cdot): \pi\in\Delta\}.
\]
We set ${\cal M}:={\cal M}^\Pi.$  For an arbitrary policy $\pi$ there exists a  stationary policy $\sigma\in\S$ such that
\begin{equation}\label{edefsig}
Q^\pi(Y\times B)=\int_Y \sigma (B|x)q^\pi(dx),\qquad Y\in {\cal X},\ B\in {\cal A},
\end{equation}
and \eqref{edefsig} implies that
\begin{equation}\label{eqpqs}
Q^\sigma(\cdot)=Q^\pi(\cdot);
\end{equation}
see \cite[Lemmas 4.1, 4.2]{FR}. Therefore,
\begin{equation}\label{emmummu}
{\cal M}^{\S}={\cal M},
\end{equation}
and this set is convex; \cite[Cor. 4.3]{FR}. These properties imply the corresponding properties of performance sets stated in the following lemma.  Recall that the initial distribution $\mu$  is fixed.
\begin{lemma}\label{lem3.1}
 For an absorbing MDP  the equality ${\cal V}^\S={\cal V}$ holds, and this set is convex.
\end{lemma}
\begin{proof}
The lemma follows from \eqref{eqingjf}, \eqref{emmummu}, and the convexity of ${\cal M}.$
\end{proof}
For an absorbing MDP with the initial state distribution $\mu,$ for  $\pi\in\Pi, $  and for $Y\in{\cal X},$ define
\[q_n^\pi(Y):=P^\pi\{x_n\in Y\},\qquad\qquad n=0,1,\ldots\ .
\]
Then
\begin{equation}\label{eqef1}
q^\pi(Y):=\sum_{n=0}^\infty E^\pi I\{x_n\in Y\}=\sum_{n=0}^\infty P^\pi (x_n\in Y)=\sum_{n=0}^\infty q_n^\pi(Y).
\end{equation}
%
We observe that $q_0^\pi(Y)=\mu(Y)$ and
\[q_n^\pi(Y)=\int_\X P_x^\pi\{x_n\in Y\}\mu(dx), \quad Y\in{\cal X},\ \pi\in\Pi, n=0,1,\ldots\ .
\]

In particular, $q^\pi(Y)=0$ if and only if $q_n^\pi(Y)=0$ for all $n=0,1,\ldots,$ $Y\in\cal X.$ This implies that $q^\sigma\ll q^\pi$ for  policies $\pi$ and $\sigma,$ if $q_n^\sigma\ll q_n^\pi$  for all $n=0,1,\ldots,$ where the symbol $\ll$ means absolute continuity.

Observe that, for a stationary policy $\pi\in\S,$ $n=0,1,\ldots,$ and $Y\in {\cal X},$
\begin{equation}\label{eqef4}
q_{n+1}^\pi (Y)=
\int_\X\int_\A p(Y|x,a)\pi(da|x)q_n^\pi(dx).
\end{equation}
Formulae \eqref{eqef1} and \eqref{eqef4} imply that for $\pi \in \S$
\begin{equation}\label{eqef5}
q^\pi (Y)=\mu(Y)+\int_\X\int_\A p(Y|x,a)\pi(da|x)q^\pi(dx).
\end{equation}

\begin{lemma}\label{Lemma A2EF.}  For two stationary policies $\pi$ and $\sigma,$ if $ \sigma(\cdot|x)\ll\pi(\cdot|x)$ for all $x\in \X,$ then
$q_n^\sigma\ll q_n^\pi,$ for all $n=0,1,\ldots,$ and therefore $q^\sigma\ll q^\pi.$\end{lemma}
\begin{proof}  For $n=0$ the statement is obvious since   $q_0^\pi=q_0^\sigma=\mu.$  Assume that $q_n^\sigma\ll q_n^\pi$ for some $n=0,1,\ldots\ .$  Consider a measurable subset $Y$ of $\X$ such that $ q_{n+1}^\pi(Y)=0.$ In view of equation \eqref{eqef4}, this means that
\[\int_\A p(Y|x,a)\pi(da|x)=0\qquad q_n^\pi- a.e.
\]
Since $ \sigma(\cdot|x)\ll\pi(\cdot|x)$ for all $x\in \X,$  as follows from the last equality,
\[\int_\A p(Y|x,a)\sigma(da|x)=0\qquad q_n^\pi- a.e.
\]
Since the integral in the left-hand part of the last equation is nonnegative and  $q_n^\sigma\ll q_n^\pi,$
\[\int_\A p(Y|x,a)\sigma(da|x)=0\qquad q_n^\sigma- a.e.,
\]
which yields
\[q_{n+1}^\sigma (Y)=\int_\X\int_\A p(Y|x,a)\sigma(da|x)q_n^\sigma(dx)=0.
\]
Thus $q_n^\sigma\ll q_n^\pi$  for all $n=0,1,\ldots,$ which implies $q^\sigma\ll q^\pi,$ as explained before \eqref{eqef4}.
\end{proof}
\begin{lemma}\label{lem3.3EF}
For an atomless absorbing MDP,  every occupancy measure $q^\pi(dx),$ where $\pi\in\Pi,$ is atomless.
\end{lemma}
\begin{proof}
In view of \eqref{emmummu}, it is sufficient to prove the lemma for stationary policies $\pi.$  Let $\pi\in \S.$ Then $q^\pi_0=\mu$ is an atomless measure.  If    $q^\pi_n$ is atomless for some $n=0,1,\ldots,$ then formula \eqref{eqef4} implies that the measure $q^\pi_{n+1}$ is atomless.  Thus, all the measures $q^\pi_n,$ $n=0,1,\ldots,$ are atomless.  Formula~\eqref{eqef1} implies that $q^\pi$ is atomless.
\end{proof}
%

The following theorem implies that ${\cal M}^{\S}={\cal M}$ and ${\cal V}^{\S}={\cal V}$ for an absorbing MDP. For discounted MDPs this result was discovered by Borkar~\cite{Bor}; see Borkar~\cite{Bor1} and Piunovskiy~\cite{Pi98} for additional references.
\begin{theorem} {\rm (Feinberg and Rothblum~\cite[Lemma 4.2]{FR})} \label{th3.4RS}
Let $\pi$ be an arbitrary policy for an absorbing MDP. Consider a stationary policy $\sigma$ such that $\sigma(B|x)=\frac{Q^\pi(dx,B)}{Q^\pi(dx,\A)}$ for each $B\in{\cal A}.$  Then the measures $Q^\sigma$ and $Q^\pi$ coincide and therefore $v^\sigma=v^\pi.$
\end{theorem}
\section{Sufficient Conditions for  Compactness of Performance Sets}\label{s5}
We start this section with formulating sufficient conditions for compactness of the set of strategic measures ${\cal S}:= \{P^\pi:\pi\in\Pi\}$ defined on the set of all trajectories $\H_\infty$ for the given initial distribution $\mu.$  Since $\H_\infty$  is a countable product of standard Borel spaces, it is a standard Borel space.  
Let ${\cal P}(\H_\infty)$ be the set of all probability measures on $\H_\infty.$ If $\A$ is a Borel subset of a Polish space, let us consider the $ws^\infty$-topology on ${\cal P}(\H_\infty),$ which is the coarsest topology in which all the mappings $P\mapsto\int f(x_0,a_0,x_1\ldots,x_t)P(dx_0da_0dx_1\ldots dx_t)$ are continuous for all bounded Borel functions $f:\H_t\mapsto \R,$ which are continuous in $ (a_0,a_1,\ldots,a_t),$ $t=0,1,\ldots\ .$  Let us consider the following version of a condition introduced by  Sch\"al~\cite{b8}. 

\medskip

{\bf Condition}~({\bf S}).
\begin{itemize}
\item[{\rm (S1)}] The set
$\A$ is a Borel subset of a Polish space, and the sets $A(x)$ are compact for all $x\in \X,$
\item[{\rm (S2)}]
 The transition probability $p(\cdot|x,a)$ is setwise continuous in $a\in A(x);$ that is, for each bounded Borel function $f:\X\mapsto \R$ and for each $x\in \X$, the function $a\mapsto\int_\X f(y) p(dy|x,a)$ is continuous  on $A(x),$
 \item[{\rm (S3)}] For each $x\in \X$ and $i=1,\ldots,N,$ the reward function $r^{(i)}(x,a)$ is continuous in $a\in A(x).$
\end{itemize}

\begin{theorem}\label{tschal} {\rm (Balder~\cite{Ba}, Nowak~\cite{No}, Sch\"al~\cite{b8})}.
If assumptions  (S1) and (S2) hold, then  the set of strategic measures ${\cal S}=\{P^\pi:\, \pi\in\Pi\}$   is a compact subset of ${\cal P}(\H_\infty)$ endowed with the $ws^\infty$-topology.
\end{theorem}
\begin{corollary} \label{cafin} Consider a uniformly absorbing MDP. If Condition~(S) holds, then the performance
set ${\cal V}$ is compact. 
\end{corollary}
\begin{proof}  Let the $ws^\infty$-topology be fixed on ${\cal P}(\H_\infty).$  Since ${\cal V}=V({\cal S}),$ where $V:{\cal S}\mapsto \R^N$ with $V(P^\pi):=v^\pi$ for all $\pi\in\Pi,$  the corollary follows from the continuity of $V,$ which is established in the rest of this proof.

Let us set $r^{(i)}({\bar x},{\bar a})=0$ for all $i=1,\ldots,N.$  This change affects neither the values of $v^\pi$ nor the validity of (S3).
Let $v^{(i),\pi}$ be the $i$th coordinate of the performance vector $v^\pi,$ $i=1,2,\ldots,N,$
\[v^{(i),\pi}=E^\pi\sum_{t=0}^{T^{\bar x}-1} r^{(i)}(x_t,a_t) = E^\pi\sum_{t=0}^\infty r^{(i)}(x_t,a_t),\]
where the second equality holds because the state $\bar x$ is absorbing and $r^{(i)}({\bar x},{\bar a})=0.$ Define
\[v_n^{(i),\pi}:=E^\pi\sum_{t=0}^{n-1} r^{(i)}(x_t,a_t),\qquad n=1,2,\ldots\ .
\]
Since the MDP is uniformly absorbing, $v_n^{(i),\pi}\to v^{(i),\pi}$ uniformly in $\pi$ as $n\to\infty.$

According to Yushkevich~\cite[Theorem 2]{Yu}, each function $r^{(i)},$ $i=1,\ldots,N,$ can be extended from ${\rm Gr}_\X(A)$ to $\X\times\A$ in a way that the extension is a bounded measurable function which is continuous in $a\in \A.$ By the definition of the $ws^\infty$-topology, the functions $V_n^{(i)}(P^\pi):=v_n^{(i),\pi}$ are continuous on ${\cal S}.$   Let $V^{(i)}(P^\pi)$ denote the $i$th coordinate of the vector $V(P^\pi).$ Since $V_n^{(i)}(P)\to V^{(i)}(P)$ uniformly for all $P\in {\cal S}$ and for all $i=1,\ldots,N,$ the mapping $V$ is continuous.  
\end{proof}
\begin{corollary}\label{cfinAabs}
Consider a uniformly absorbing MDP. If each set $A(x),$  is finite, $x\in \X,$ then the performance set ${\cal V}$ is compact. 
\end{corollary}
\begin{proof}
If $\A$ is a Borel subset of a Polish space, then the conclusion of the corollary follows from Corollary~\ref{cafin} since Condition  (S) holds.  The corollary follows from this fact since a standard Borel space is isomorphic to a Borel subset of a Polish space.  Indeed, let
 $\tilde\A$ be a Borel subset of a Polish space isomorphic to $\A,$ and let $g:\tilde\A\mapsto\A$ be the corresponding isomorphism.  Let us consider the MDP with the state space $\X,$ the action space $\A$  replaced with the isomorphic set $\tilde\A,$ the sets of available actions ${\tilde A}(x):=g^{-1}(A(x)),$  one-step rewards vectors ${\tilde r}(x,a):= r(x,g(a)),$ and transition probabilities  ${\tilde p}(\cdot|x,a)=p(\cdot|x,g(a)),$ where $x\in\X$ and $a\in {\tilde \A}.$  The performance sets $\cal V$ for the new and original models coincide.
The set $\cal V$ is compact since $\tilde\A$ is a Borel subset of a Polish space.
\end{proof}

\section{Submodels and Dimensionality Reduction}\label{s6}
\begin{definition}{\rm
An MDP $\{{\tilde \X},{\tilde \A},{\tilde A}(\cdot),{\tilde p}, {\tilde r}\}$ is called a {\it submodel} of the MDP $\{\X,\A, A(\cdot), p, r\},$ if ${\tilde \X}=\X,$ ${\tilde \A}=\A,$ ${\tilde p}=p,$ ${\tilde r}=r,$ and ${\tilde A}(x)\subset A(x)$ for all $x\in \X.$ }
\end{definition}
We say that a submodel is well-defined, if the set ${\rm Gr}_{\tilde \X}({\tilde A})$ is a Borel subset of $\tilde\X\times \tilde\A$ and there exists at least one deterministic policy (selector) in the submodel.  The existence of a selector usually follows from measurable selection theorems. According to the Arsenin-Kunugui selection theorem (Kechris~\cite[Th. 18.18]{Ke}), a measurable selector $\phi:\tilde{\X}\mapsto \tilde{\A},$ such that $\phi(s)\in{\tilde A}(x)$ for all $x\in {\tilde\X},$ exists, if $\tilde\A$ is a Borel subset of a Polish space, the set ${\rm Gr}_{\tilde \X}({\tilde A})$ is a Borel subset of $\tilde\X\times \tilde\A,$ and each set $\tilde A(x)$ is a union of  a countable number of nonempty compact subsets of $\tilde\A.$  In addition, this theorem claims that under these assumptions  the projection of any Borel subset of   ${\rm Gr}_{\tilde\X}(\tilde{A})$ onto $\tilde\X$ is a Borel subset of $\tilde\X.$  If ${\rm Gr}_{\tilde\X}({\tilde A})$ is a Borel subset of $\tilde\X\times\tilde\A$ and each set ${\tilde A}(x),$ $x\in \tilde\X,$ is nonempty and finite or countable, then the  Arsenin-Kunugui   theorem implies that the submodel is well-defined and the projection of any Borel subset of ${\rm Gr}_{\tilde\X}({\tilde A})$ onto $\tilde\X$ is a Borel subset of $\tilde\X.$

It is obvious that a submodel inherits many properties of the MDP including atomless, absorbing, and uniformly absorbing properties.  In addition, ${\tilde {\cal V}}\subset {\cal V},$ where ${\tilde {\cal V}}$ is the performance set for the submodel.
\begin{lemma}\label{lredcountA}
Consider an absorbing atomless MDP.  Then for every  $v\in{\cal V}$ there exists a submodel with finite or countable action sets ${\tilde A}(x),$ $x\in \X,$ such that, for some stationary policy $\pi$ for this submodel,  $v^\pi=v$ and $\pi(a|x)>0$ for all $x\in \X$ and all $a\in {\tilde A}(x).$
\end{lemma}
\begin{proof}
According to Feinberg and Piunovskiy ~\cite[Theorem 2.1]{b5}, there exists a nonrandomized Markov policy $\phi=(\phi_0,\phi_1,\ldots )$ such that
$v^\phi=v.$  Let us define the nonempty sets $A_\phi(x):=\cup_{n=0}^\infty\{\phi_n(x)\},$ which are either countable or finite. Observe that the set  ${\rm Gr}_\X(A_\phi)=\cup_{n=0}^\infty {\rm Gr}_\X(\phi_n)$ is Borel because the graph of a Borel function $\phi_n$ is a Borel set; see e.g., Bertsekas and Shreve~\cite[Cor. 7.14.1]{b3}.

In view of Theorem~\ref{th3.4RS}, there is a stationary policy $\pi$ such that $\pi(\cdot|x)$ is concentrated on $A_\phi(x)$ and $v^\pi=v^\phi=v.$ Let ${\tilde A(x)}=\{a\in A_\phi(x):\, \pi(a|x)>0\},$ $x\in \X.$
Since $\pi(A_\phi(x)|x)=1,$ then ${\tilde A}(x)\ne\emptyset$ for all $x\in \X.$   The set ${\rm Gr}_\X({\tilde A})$ is Borel because ${\rm Gr}_\X({\tilde A})=\{(x,a)\in \X\times\A:\, G(x,a)>0 \},$ where $G(x,a)=\sum_{n=0}^\infty \pi(\phi_n(x)|x)I\{(x,a)\in {\rm Gr}_\X(\phi_n)\},$ and because the functions  $I\{(x,a)\in {\rm Gr}_\X(\phi_n)\}$ and $\pi(\phi_n(x)|x)$ are Borel-measurable, where  the measurability of the function $I\{(x,a)\in {\rm Gr}_\X(\phi_n)\}$ follows from the measurability of the sets ${\rm Gr}_\X(\phi_n)\subset\X\times\A$ and the measurability of the function $\pi(\phi_n(x)|x)$ follows from  Bertsekas and Shreve~\cite[Cor. 7.26.1]{b3}.
\end{proof}


\begin{theorem}\label{toptimsingle} Consider a uniformly absorbing atomless MDP. Suppose that $N=1$ and there exists a stationary policy $\sigma^*$ such that $v^{\sigma^*}=\sup_{\sigma\in \S} v^\sigma.$  For $v:=v^{\sigma^*}\in {\cal V}$ consider a stationary policy $\pi$ and a  submodel with action sets ${\tilde A}(\cdot),$ whose existence is stated in Lemma~\ref{lredcountA}.  Then $v^{\pi^*}=v^{\sigma^*}$ for each policy $\pi^*$ in this submodel.
\end{theorem}
\begin{proof}  Let $N=1.$
In view of Theorem~\ref{th3.4RS},  for an absorbing MDP  $\sup_{\tilde\pi\in \S} v^{\tilde\pi}= \sup_{\tilde\pi\in \Pi} v^{\tilde\pi}$ and  $v^\sigma= \sup_{\tilde\pi\in \Pi} v^{\tilde\pi}$ for some policy $\sigma$ if and only  $v^{\sigma^*}= \sup_{\tilde\pi\in \S} v^{\tilde\pi}$  for some  stationary policy $\sigma^*.$ Recall  that  $\sup_{\tilde\pi\in \S} v^{\tilde\pi}=\sup_{\phi\in \dF} v^\phi;$ see Feinberg~\cite{F92}.


%
For an arbitrary policy $\sigma\in\Pi,$ let $X^\sigma$ be the set of initial states $x\in\X,$ for which the expected initial rewards $v^\sigma(x)$ are well-defined, that is,
 \begin{equation}\label{eqdefofXpi}
X^\sigma=\{x\in\X:\,v^\sigma_+(x)<+\infty\}\cup \{x\in\X:\,v^\sigma_-(x)<+\infty\}.
 \end{equation}
 In view of the Ionescu Tulcea theorem~\cite[Sect. V.1]{Neveu}, the functions $v^\sigma_+(x)$ and $v^\sigma_-(x)$ are Borel measurable.  Therefore, the set $X^\sigma$ is Borel as the union of two Borel sets.

For $x\in \X,$ $a\in \tilde A(x),$ and for a Borel function $f:\X\mapsto\R^1,$ let us denote
\[{\bf T}^a f(x):= r(x,a)+\int_\X f(y) p(dy|x,a),\qquad \quad x\in \X, a\in \tilde A(x).\]
This value is well-defined if either $\int_\X f^+(y) p(dy|x,a)<+\infty$ or $\int_\X f^-(y) p(dy|x,a)<+\infty.$  

Let $\sigma$ be a stationary policy in the submodel with action sets $\tilde{A}(\cdot).$  Then  ${\bf T}^a v^\sigma(x)$ is well-defined for $x\in X^\sigma$ and $a\in\tilde{A}(x),$ where  the Borel set $X^\sigma$ is defined in \eqref{eqdefofXpi}. Indeed,
\[v^\sigma_+(x)=\sum_{a\in\tilde A(x)} \sigma(a|x)\{\ r^+(x,a)+\int_\X v_+^\sigma(y) p(dy|x,a)\}<+\infty,\qquad x\in X^\sigma,\]
and
\[v^\sigma_-(x)=\sum_{a\in\tilde A(x)} \sigma(a|x)\{\ r^-(x,a)+\int_\X v_-^\sigma(y) p(dy|x,a)\}<+\infty,\qquad x\in X^\sigma.\]
Therefore,
\begin{equation}\label{eqpisec6}v^\sigma(x)=v^\sigma_+(x)-v^\sigma_-(x)=\sum_{a\in\tilde A(x)} \sigma(a|x){\bf T}^av^\sigma(x),\qquad x\in X^\sigma,
\end{equation}
 and ${\bf T}^av^\sigma(x)$ is well-defined for $x\in X^\sigma$ and $a\in\tilde A(x)$ if   $\sigma(a|x)>0.$

 Observe that for an absorbing MDP $q^\sigma(\X\setminus X^\sigma)=0,$   which is equivalent to $q^\sigma(\X)=q^\sigma(X^\sigma).$     Indeed, if $q^\sigma(\X\setminus X^\sigma)>0,$ then, in view of \eqref{eqef1}, $P^\sigma\{x_n\in\X\setminus X^\sigma\}>0$ for some $n=0,1,\ldots\ .$  This implies that either $v_+^\sigma=+\infty$ or $v_-^\sigma=+\infty.$  This conclusion contradicts to the assumptions that the MDP is absorbing and the reward function $r$ is bounded.

In particular, for $\sigma=\pi,$ where the policy $\pi$ is defined in Lemma~\ref{lredcountA},
\begin{equation}\label{eqXmxpi0} q^\pi(\X\setminus X^\pi)=0.\end{equation}

By Lemma~\ref{lredcountA}, $v^\pi=v=v^{\sigma^*}.$ Consider the sets
\[ X^>:=\{x\in X^\pi:\, {\bf T}^a v^\pi(x)>v^\pi(x)\ {\rm for\ some\ } a\in {\tilde A}(x)\},\]
\[ X^<:=\{x\in X^\pi:\, {\bf T}^a v^\pi(x)< v^\pi(x)\ {\rm for\ some\ } a\in {\tilde A}(x)\},\]
\[ X^=:=\{x\in X^\pi:\, {\bf T}^a v^\pi(x)= v^\pi(x)\ {\rm for\ all\ } a\in {\tilde A}(x)\}.\]

 The sets $X^>,$ $X^<,$ and $X^=$ are Borel. Indeed,  the set $X^>$ is a projection of the Borel set $Y(\pi):=\{(x,a)\in {\rm Gr}_{X^\pi}(\tilde{A}):\, {\bf T}^a v^\pi(x)>v^\pi(x)\}$ onto $X^\pi.$ In addition, each action set ${\tilde A}(x),$ $x\in \X,$ is finite or countable.  Therefore, in view of the  Arsenin-Kunugui   theorem, the set $X^>$ is Borel  and there exists a Borel mapping $\varphi^*:X^>\mapsto \A$ such that $\varphi^*(x)\in \tilde A(x)$ and ${\bf T}^{\varphi^*(x)}v^\pi(x)>v^\pi(x)$ for all $x\in X^>.$
 The set $X^<$ is Borel because it is a projection of the Borel set $\{(x,a)\in {\rm Gr}_{X^\pi}(\tilde{A}):\, {\bf T}^a v^\pi(x)<v^\pi(x)\}$ onto $\X.$
 Thus, $X^==X^\pi\setminus (X^>\cup X^<)$ is a Borel set too.

 Observe that \begin{equation}\label{eqYst00}
  q^\pi(X^<) =q^\pi(X^>)=0.\end{equation}
To prove the second equality in \eqref{eqYst00}, suppose that $q^\pi(X^>)>0.$  Therefore, $q_n^\pi(X^>)=P^\pi\{x_n\in X^>\}>0$ for some $n=0,1,\ldots\ .$ For the Borel mapping $\varphi^*$ described in the previous paragraph,   consider a randomized Markov policy $\pi^\prime$
\[
 \pi^\prime_t(B|x)=\begin{cases} I\{\varphi^*(x)\in B\}, & {\rm if}\ t=n\ {\rm and}\ x\in X^>,\\
 \pi(B|x), &{\rm otherwise},
 \end{cases}
 \]
 where $B\in{\cal A}$ and $t=0,1,\ldots\ .$
  Straightforward calculations imply that
\[{ v}^{\pi^\prime}-{ v}^\pi=\int_{X^>} [{\bf T}^{\varphi^*(x)}{v}^\pi(x)-{ v}^\pi(x)]q_n^\pi(dx)>0,
\]
which contradicts $v^\pi=v^{\sigma^*}=\sup_{\sigma\in \S} v^\sigma =\sup_{\sigma\in \Pi} v^\sigma\ge v^{\pi^\prime},$ where the last equality follows from Theorem~\ref{th3.4RS}.   Thus,  the second equality in \eqref{eqYst00} is proved.

The equality $q^\pi(X^<)=0$ holds because the inequality $q^\pi(X^<)>0$ is impossible.  Indeed, if $q^\pi(X^<)>0,$ then $q^\pi(X^<\setminus X^>)=q^\pi(X^<)>0$ because $q^\pi(X^>)=0.$  Therefore,
\[0=\int_{X^<\setminus X^>} (v^\pi(x)-v^\pi(x))q^\pi(dx)=\int_{X^<\setminus X^>}\sum_{a\in\tilde A(x)}\pi(a|x)({\bf T}^av^\pi(x)-v^\pi(x))q^\pi(dx)<0,
\]
where the second equality follows from \eqref{eqpisec6} and the inequality holds because an integral of a negative  function on a set with a positive measure is negative.  The function is negative because $\pi(a|x)>0$ for all $a\in \tilde A(x),$ the difference in the second integral is nonpositive for all $a\in \tilde A(x),$ and this difference is negative for some
$a\in \tilde A(x),$ where $x\in {X^<\setminus X^>}.$  Equalities  \eqref{eqYst00} are proved.

The equality $v^{\pi^*}=v^{\sigma^*}$ holds for every policy $\pi^*$ in the submodel with action sets $\tilde{A}(\cdot)$ if and only if $v^\sigma=v^\pi$ for every stationary policy $\sigma$ in this submodel.  This is true in view of Theorem~\ref{th3.4RS} and because $v^\pi=v^{\sigma^*}=v.$
Let $\sigma$ be a stationary policy for the submodel with action sets ${\tilde A}(\cdot).$ To complete the proof, we  show in the rest of the proof that $v^\sigma=v^\pi.$

Since $\sigma(\cdot|x)\ll \pi (\cdot|x)$ for all $x\in \X,$   Lemma~\ref{Lemma A2EF.} and formulae \eqref{eqXmxpi0}, \eqref{eqYst00} imply that $ q^\sigma(\X\setminus X^=) =0.$ Let $\sigma^{n,\pi}$ be the policy that follows $\sigma$ at times $t=0,1,\ldots,n-1$ and follows $\pi$ at $t=n,n+1,\ldots.$ In particular, $\sigma^{0,\pi}=\pi.$  Induction arguments imply that
\begin{equation}\label{eqingsnpi} v^{\sigma^{n,\pi}}=v^\pi,\qquad\qquad n=0,1,\ldots\ .\end{equation}
Indeed, for $n=0$ formula~\eqref{eqingsnpi}  holds because $\sigma^{0,\pi}=\pi.$  If \eqref{eqingsnpi} holds for some $n=0,1,\ldots\,$ then
\[v^{\sigma^{n+1,\pi}}(x)=\sum_{a\in \tilde{A}(x)} \sigma(a|x)T^av^\pi(x)=v^\pi(x),\qquad x\in X^=,\]
and
\[
v^{\sigma^{n+1,\pi}}=\int_\X v^{\sigma^{n+1,\pi}}(x)\mu(dx)=\int_{X^=} v^\pi(x)\mu(dx)=\int_\X v^\pi(x)\mu(dx)=v^\pi,
\]
where the last equalities hold because $\mu(\X\setminus X^=)=0$ since $\mu\ll q^\pi$ and $q^\pi(\X\setminus X^=)=0$ in view of  \eqref{eqXmxpi0} and \eqref{eqYst00}. Formula \eqref{eqingsnpi} is proved.

Since the MDP is uniformly absorbing,
\[ \lim_{n\to\infty} E^{\sigma^{n,\pi}}\sum_{t=n}^\infty I\{t<T^{\bar x}\}=\lim_{n\to\infty}\sup_{\tilde{\pi}\in M}E^{\tilde\pi} \sum_{t=n}^\infty I\{t<T^{\bar x}\}=0.\]
Since the reward function $r$ is bounded,
\[ \lim_{n\to\infty}E^{\sigma^{n,\pi}}\sum_{t=n}^\infty r(x_t,a_t)=0.
\]
Therefore,
\[
v^\sigma=\lim_{n\to\infty}E^\sigma\sum_{t=0}^{n-1} r(x_t,a_t)=\lim_{n\to\infty}E^\sigma\sum_{t=0}^{n-1} r(x_t,a_t)+\lim_{n\to\infty}E^{\sigma^{n,\pi}}\sum_{t=n}^\infty r(x_t,a_t)=\lim_{n\to\infty}v^{\sigma^{n,\pi}}=v^\pi,\]
where the last equality follows from \eqref{eqingsnpi}.
\end{proof}
The following lemma is correct without the assumption that the MDP is atomless.  However, we need it only for an atomless MDP in this paper, and for an atomless MDP the proof follows directly from Theorem~\ref{toptimsingle}.

\begin{corollary}\label{coropt2}
Consider a uniformly absorbing atomless MDP with $N=1.$ For every extreme point $v\in\cal V$ of the set $\cal V$  there exists a deterministic policy $\phi$ such that $v^\phi=v.$
\end{corollary}
\begin{proof}
Since $N=1,$ the closure of the convex set $\cal V$ is a bounded interval on the line.  Therefore, there could be at most two extreme points $v_*:=\inf_{\pi\in\Pi} v^\pi$ and
$v^*:=\sup_{\pi\in\Pi} v^\pi.$  Let us consider $v=v^*.$ Theorem~\ref{th3.4RS} implies that  $v=\sup_{\pi\in \S} v^\pi.$ According to Theorem~\ref{toptimsingle},  $v^\phi=v$ for every deterministic policy $\phi$ in the submodel, whose existence is stated in Lemma~\ref{lredcountA}. The change $r:=-r$ reduces the case  $v=v_*$ to the case  $v=v^*.$
\end{proof}

For $i=1,\ldots,N,$ let us denote by $b_{-i}$ the projection of $b\in \R^N$ to $\R^{N-1}$ obtained by removing the $i$-th coordinate of the vector $b.$ Also, $\langle\cdot,\cdot\rangle$ denotes the scalar product of two vectors.  
\begin{definition}{\rm We say that \emph{a point $v\in {\cal V}$ allows the dimensionality reduction,} if there is a coordinate $i=1,2,\ldots,N,$ a vector $b\in\R^{N-1},$  a constant $d,$ and a submodel  $\{{\X}, \A,{\tilde A}(\cdot), p, r\}$ of the original MDP such that $v\in {\tilde {\cal V}},$ where $\tilde{\cal V}$ is the performance set for all policies in the submodel, and
\begin{equation}{\tilde v}^{(i)}= d+\langle b,{\tilde v}_{-i}\rangle \qquad {\rm for\ all}\qquad {\hat v}\in{\tilde {\cal V}}.\label{edimred}\end{equation}
}
\end{definition}
The following theorem plays an important role in the proof of Theorem~\ref{tgmain}.  Recall that  $\partial (C)$ is the boundary of a bounded convex set $C\in \R^n,$ $n=1,2,\ldots\ .$
\begin{theorem}\label{tdimred} {\rm (Dimensionality reduction).}
For a uniformly absorbing atomless MDP, 
each point on the boundary of ${\cal V}$ allows the dimensionality reduction.
\end{theorem}
\begin{proof}
Let $v^*\in \partial( {\cal V}).$
 Let $\langle{\tilde b}, v\rangle={\tilde d}$ be a supporting hyperplane at the point $v^*$ to the convex set $\cal V$  such that $\langle{\tilde b}, v\rangle \le {\tilde d}$ for all $v\in {\cal V}$ and $\langle{\tilde b}, v^*\rangle={\tilde d},$ where ${\tilde b}^{(i)}\ne 0$ for at least one $i=1,\ldots,N.$ Let us define the one-step reward function
\[{\tilde r}(x,a):=\langle{\tilde b},r(x,a)\rangle,\qquad\qquad x\in \X, a\in A (x).
\]
Let ${\tilde v}^\sigma$   be the expected total rewards for this reward function, initial distribution $\mu,$ and a policy $\sigma.$ Then ${\tilde v}^\sigma=\langle {\tilde b}, v^\sigma\rangle.$

Since $v^*\in {\cal V},$ then $v^*=v^{\sigma^*}$ for a stationary policy  $\sigma^*\in \S.$ Using Lemma~\ref{lredcountA}, consider  the corresponding submodel with finite or countable action sets ${\tilde A}(\cdot)$ and a stationary policy $\pi$ for this submodel, where ${\tilde {\cal V}}$ is the performance set for the submodel.  
In particular, $v^\pi=v^*\in {\tilde {\cal V}}.$ Note that  ${\tilde v}^\pi={\tilde d}=\sup_{v\in\cal V} v=\sup_{\sigma\in \S} {\tilde v}^\sigma.$  In view of Theorem~\ref{toptimsingle},
\begin{equation}\label{eq24}{\tilde v}^\pi=\langle {\tilde b},{\hat v}\rangle \qquad {\rm for\ all}\qquad  {\hat v}\in {\tilde {\cal V}}.\end{equation}
 Formula~\eqref{eq24} implies \eqref{edimred} with $d:=\langle {\tilde b}, v^*\rangle/{\tilde b}^{(i)}$ and $b:=-{\tilde b}_{-i}/{\tilde b}^{(i)},$ where $i=1,\ldots,N$ with $\tilde{b}^{(i)}\ne 0$ and  $\tilde{b}^{(i)}$ is the $i^{\rm th}$ coordinate of the vector $\tilde{b}.$ 
\end{proof}

\section{An MDP Defined by Two Deterministic Policies}\label{s7}
Let $\phi^0$ and $\phi^1$ be two deterministic policies.  These two policies are considered to be fixed within this section.  Let us define action sets $A^*(x):=\{\phi^0(x), \phi^1(x)\}$ and consider an MDP, which is the submodel obtained from the original MDP by narrowing the action sets $A(x)$ to $A^*(x)$ for all $x\in \X.$  We say that this MDP is defined by the deterministic policies $\phi^0$ and $\phi^1.$

Consider  the stationary policy $\pi^*:$
\begin{equation}\label{eqdefpist}\pi^*(B|x):=\frac{1}{2}[I\{\phi^0(x)\in B\}+I\{\phi^1(x)\in B\}],\qquad B\in{\cal A}, x\in\X,\end{equation}
which  averages the deterministic policies $\phi^0$ and $\phi^1.$   We denote by $q$ the occupancy measure $q^{\pi^*}$ on $\X,$
\begin{equation} \label{edrfqEFEF}q(Y):= q^{\pi^*}(Y),\qquad\qquad\qquad Y\in{\cal X.} \end{equation}
 \begin{lemma}\label{ldominphi01}
 $q^\gamma\ll q$ for every stationary policy $\gamma$ for the MDP defined by two deterministic policies $\phi^0$ and $\phi^1.$
 \end{lemma}
 \begin{proof} This lemma follows from Lemma~\ref{Lemma A2EF.} since
  $\gamma(\cdot|x)\ll \pi^*(\cdot|x),$ $x\in \X,$  for each stationary policy $\gamma$ for the MDP defined by two deterministic policies $\phi^0$ and $\phi^1.$
\end{proof}

The following lemma provides a useful inequality.

\begin{lemma}\label{lobvineq}  For every stationary policy $\gamma$ for the MDP defined by two deterministic policies $\phi^0$ and $\phi^1,$ the inequality $E^\gamma f(x_t)\le 2^tE^{\pi^*} f(x_t)$ holds
for an arbitrary nonnegative measurable function $f$ and for each $t=0,1,\ldots\ .$
\end{lemma}
\begin{proof}  The proof is based on the induction in $t.$  Since $E^\gamma f(x_0)=\int_\X f(x)\mu(dx)$ for every stationary policy $\gamma,$ the inequality holds for $t=0$ in the form of the equality. Let this inequality  hold for some $t=0,1,\ldots\ .$ Then
\begin{equation}\label{econdexp}
\begin{aligned}E^\gamma[f(x_{t+1})|x_t]=\int_\X f(x)\sum_{i=0}^1\gamma(\phi^i(x_t)|x_t)p(dx|x_t,\phi^i(x_t))\le \int_\X f(x)\sum_{i=0}^1p(dx|x_t,\phi^i(x_t))\\ 
=2\int_\X f(x)\sum_{i=0}^1\frac{1}{2}p(dx|x_t,\phi^i(x_t))=2E^{\pi^*}[f(x_{t+1})|x_t],\qquad\qquad\qquad\qquad\qquad
\end{aligned}
\end{equation}
where the first and the last equalities follow from the definitions of strategic measures, and the inequality and the second equality are obvious.  Therefore,
$E^\gamma f(x_{t+1})  =  E^\gamma E^\gamma[f(x_{t+1})|x_t]\le 2E^\gamma E^{\pi^*}[f(x_{t+1})|x_t] \le  2^{t+1}E^{\pi^*} E^{\pi^*}[f(x_{t+1})|x_t]=2^{t+1}E^{\pi^*} f(x_{t+1}),
$
where the first and the last equalities follow from the  definition of a conditional expectation, the first inequality follows from \eqref{econdexp}, and the second inequality follows from the induction assumption.
\end{proof}

\begin{corollary}\label{coropt1EFS6} For $t=0,1,\ldots $ and for every $Y\in\cal X,$ the inequality
$q_t^\gamma(Y)\le 2^tq_t(Y)$ holds for every stationary policy for the MDP defined by two deterministic policies $\phi^0$ and $\phi^1.$
\end{corollary}
\begin{proof}
The corollary follows from Lemma~\ref{lobvineq} applied to the function $f(x)= I\{x\in Y\},$ $x\in\X.$
\end{proof}

For two stationary policies $\pi$ and $\sigma$  for the MDP defined by two deterministic policies $\phi^0$ and $\phi^1,$  let
\begin{equation}\label{enonotcoinEF}X(\pi,\sigma):=\{x\in \X: \pi(\cdot|x)=\sigma(\cdot|x)\}=\{x\in \X: \pi(\phi^0(x)|x)=\sigma(\phi^0(x)|x)\}\end{equation}
be the set of states on which $\pi$ and $\sigma$ choose the same decisions.  In view of the last equality, this set is measurable.
\begin{lemma}\label{lem6.1} Consider a uniformly absorbing MDP. If $q(\X\setminus X(\pi,\sigma))=0,$ then $q^\pi=q^\sigma,$ where $\pi$ and $\sigma$ are arbitrary stationary policies in the MDP defined by two deterministic policies $\phi^0$ and $\phi^1$.
\end{lemma}
\begin{proof} As follows from \eqref{eqdefpist}, $\pi(\cdot|x)\ll \pi^*(\cdot|x)$ and $\sigma(\cdot|x)\ll \pi^*(\cdot|x)$ for all $x\in \X.$   Lemma~\ref{Lemma A2EF.} implies that $q^\pi\ll q$ and $q^\sigma\ll q.$  Therefore, $q^\pi(\X\setminus X(\pi,\sigma))=0$ and $q^\sigma(\X\setminus X(\pi,\sigma))=0$ if $q(\X\setminus X(\pi,\sigma))=0.$ Thus, the set of states, on which the stationary policies $\pi$ and $\sigma$ make different decisions, will be visited with zero probability when each of these policies is used. \end{proof}

Let $d_{TV}(\eta_1,\eta_2)$ denote  the distance in total variation between two finite measures defined on the same measurable space;
see e.g., \cite[Section 2]{FKZ14} or \cite{FKZUFL} for details on definitions and properties of distances in  total variation for finite measures. Since $q^\pi(dx)=Q^\pi(dx,\A)$ for an arbitrary policy $\pi,$ then $d_{TV}(q^\pi,q^\sigma)\le d_{TV}(Q^\pi,Q^\sigma)$ for two policies $\pi$ and $\sigma.$
 As follows from Lemma~\ref{lem6.1},  $q(\X\setminus X(\pi,\sigma))=0$ implies that $q^\pi=q^\sigma.$  The following theorem, which is the main result of this section, demonstrates that the value of $q(\X\setminus X(\pi,\sigma))$ characterizes how close the measures $Q^\pi$ and $Q^\sigma$ are.  
\begin{theorem}\label{th6.2} Consider a uniformly absorbing MDP.
Let $\pi$ and $\sigma$ be two stationary policies  for the MDP defined by two deterministic policies $\phi^0$ and $\phi^1.$ Then for every $\epsilon >0$ there exists $\delta>0$ such that, if  $q(\X\setminus X(\pi,\sigma))\le \delta,$ then  $d_{TV}(Q^\pi,Q^\sigma)\le \epsilon.$
\end{theorem}
\begin{proof} Let us fix an arbitrary $\epsilon>0.$ In this proof $\gamma$ is always a policy that is equal either to $\pi$ or to $\sigma.$ In other words, $\gamma\in\{\pi,\sigma\}.$

   We prove first the existence of $\delta>0$ such that, if  $q(\X\setminus X(\pi,\sigma))\le \delta,$ then $d_{TV}(q^\pi,q^\sigma)\le \epsilon.$    This claim follows from the following fact.  There exist a constant $\delta>0$ and measures ${\bar q}^\gamma$ and ${\hat q}^\gamma$ on $(\X,{\cal X})$ such that the inequality $q(\X\setminus X(\pi,\sigma))\le \delta$ implies the correctness of the following statements:
(i) $q^\gamma= {\bar q}^\gamma + {\hat q}^\gamma,$ (ii) ${\hat q}^\gamma(\X)\le \epsilon/2,$ and (iii) ${\bar q}^\pi = {\bar q}^\sigma.$  If this is true, then $d_{TV}(q^\pi,q^\sigma)= d_{TV}(\hat{q}^\pi,\hat{q}^\sigma)  \le\epsilon.$

Let us construct a positive constant $\delta$ and  measures
${\bar q}^\gamma$ and ${\hat q}^\gamma$ on $(\X,{\cal X})$ satisfying properties (i)--(iii).
We denote by ${\bar T}^Y:=\min\{t=0,1,\ldots: x_t\notin Y\}$ the first time the process leaves the set $Y\in\cal X$ and  define the measure
\[{\bar q}^\gamma(C)=E^\gamma \sum_{t=0}^\infty I\{x_t\in C\}I\{\bar{T}^{X(\pi,\sigma)}>t\},\qquad\qquad C\in{\cal X}.\]
Since the stationary policies $\pi$ and $\gamma$ coincide on the set $X(\pi,\sigma),$
\[{\bar q}^\pi = {\bar q}^\sigma.\]
Thus, (iii) holds. Since the MDP is uniformly absorbing, there exist $\ell=1,2,\ldots$ such that for every stationary policy $\pi^\prime$
\begin{equation}\label{efirsteste}D^{\pi^\prime}_1:= E^{\pi^\prime} \sum_{t=\ell}^\infty I\{x_t\in \X\}\le \epsilon/4.\end{equation}
In particular, \eqref{efirsteste} holds for $\pi^\prime=\gamma.$

Our next step is to show that there exists $\delta>0$ such that, if $q(\X\setminus X(\pi,\sigma))\le \delta,$ then
\begin{equation} \label{estD1EEF}
D^\gamma_2:=E^\gamma \sum_{t=0}^{\ell-1} I\{x_t\in \X\}I\{{\bar T}^{X(\pi,\sigma)}\le t\}\le\epsilon/4.
\end{equation}
Indeed, by exchanging the summation and expectation in \eqref{estD1EEF}, we have
\begin{equation}\label{estD1EEFPg}
D^\gamma_2=\sum_{t=0}^{\ell-1} P^\gamma\{x_t\in \X, {\bar T}^{X(\pi,\sigma)}\le t\}.
\end{equation}
Observe that for $t=0,1,\ldots$
\begin{eqnarray}\label{estD1EEFP2a}
 P^\gamma\{x_t\in \X, {\bar T}^{X(\pi,\sigma)}\le t\}\le\sum_{s=0}^tP^\gamma\{x_t\in\X, x_s\in\X\setminus X(\pi,\sigma)\}\le \sum_{s=0}^t q_s^\gamma(\X\setminus X(\pi,\sigma)).
\end{eqnarray}
In view of Corollary~\ref{coropt1EFS6},
\begin{equation}\label{estD1EEFP2}
\sum_{s=0}^t q_s^\gamma(\X\setminus X(\pi,\sigma))\le \sum_{s=0}^t 2^sq_s(\X\setminus X(\pi,\sigma))\le 2^t \sum_{s=0}^t q_s(\X\setminus X(\pi,\sigma))\le 2^t  q(\X\setminus X(\pi,\sigma)).
\end{equation}
Formulae (\ref{estD1EEFPg}--\ref{estD1EEFP2}) imply that $D_2^\gamma\le 2^{\ell} q(\X\setminus X(\pi,\sigma)).$  Thus, \eqref{estD1EEF} holds with $\delta=2^{-(\ell+2)}\epsilon.$

Let us define the measures
\[ {\hat q}^\gamma(C)=E^\gamma \sum_{t=0}^\infty I\{x_t\in C\}I\{\bar{T}^{X(\pi,\sigma)}\le t\},\qquad\qquad \qquad\qquad C\in\cal X. \]
Then $q^\gamma={\bar q}^\gamma+{\hat q}^\gamma.$ Thus, (i) holds.   Let $\delta=2^{-(\ell+2)}\epsilon.$ For $\gamma\in\{\pi,\sigma\}$
\begin{eqnarray*}{\hat q}^\gamma(\X)=E^\gamma \sum_{t=0}^\infty I\{x_t\in \X\}I\{\bar{T}^{X(\pi,\sigma)}\le t\}\le E^\gamma \sum_{t=0}^{\ell-1} I\{x_t\in \X\}I\{\bar{T}^{X(\pi,\sigma)}\le t\}+E^\gamma \sum_{t=\ell}^\infty I\{x_t\in \X\}\le\epsilon/2,
\end{eqnarray*}
where the last inequality follows from \eqref{estD1EEF} and \eqref{efirsteste}. Thus, (ii) holds. In view of (i)--(iii), $d_{TV}(q^\pi,q^\sigma)\le \epsilon.$

Let us prove the inequality $d_{TV}(Q^\pi,Q^\sigma)\le \epsilon.$  To do this, we consider the measures ${\bar Q}^\gamma$ and ${\hat Q}^\gamma$ on $(\X\times\A,{\cal X}\times\cal A )  $ defined by
\begin{eqnarray*}
{\bar Q}^\gamma(C\times B)=E^\gamma \sum_{t=0}^\infty I\{x_t\in C, a_t\in B\}I\{\bar{T}^{X(\pi,\sigma)}>t\},\qquad\qquad C\in{\cal X}, B\in{\cal A},\\
{\hat Q}^\gamma(C\times B)=E^\gamma \sum_{t=0}^\infty I\{x_t\in C, a_t\in B\}I\{\bar{T}^{X(\pi,\sigma)}\le t\},\qquad\qquad C\in{\cal X}, B\in{\cal A}.
\end{eqnarray*}
These two measures obviously satisfy the following properties: (${\rm i^*}$) $Q^\gamma={\bar Q}^\gamma+{\hat Q}^\gamma,$  (${\rm ii^*}$) ${\hat Q}^\gamma(\X\times\A)={\hat q}^\gamma(\X) \le \epsilon/2,$ (${\rm iii^*}$) ${\bar Q}^\pi = {\bar Q}^\sigma.$ Properties  (${\rm i^*}$)--(${\rm iii^*}$) imply $d_{TV}(Q^\pi,Q^\sigma)\le \epsilon.$
\end{proof}

Let $\Vert\cdot\Vert$ be the Euclidean norm in $\R^N.$ The following corollary follows from Theorem~\ref{th6.2}.
\begin{corollary}
Let $\pi$ and $\sigma$ be two stationary policies  in the MDP defined by two deterministic policies $\phi^0$ and $\phi^1.$ Then for every $\epsilon >0$ there exists $\delta>0$ such that the inequality $q(\X\setminus X(\pi,\sigma))\le \delta$ implies that  $\Vert v^\pi-v^\sigma\Vert\le \epsilon.$
\end{corollary}
\begin{proof}

Let $K$  be a finite positive constant  satisfying $K\ge |r^{(n)}(x,a)|$ for all $n=1,\ldots,N,$ $x\in \X,$ and $a\in A(x).$ Then the corollary follows from 
Theorem~\ref{th6.2} applied to the constant $\epsilon_1:={\epsilon}/{(KN^\frac{1}{2})}$ instead of $\epsilon.$
\end{proof}
\section{Path Connectedness of the Set of Occupancy Measures Generated by Deterministic Policies}\label{s8}

 We recall that a subset $E$ of a topological space is called path-connected, if for every two  points $e_0, e_1\in E$ there exists a continuous function $g:[0,1]\mapsto E$ such that $g(0)=e_0$ and $g(1)=e_1.$ A set is called connected, if it cannot be partitioned into two nonempty subsets which are open in the relative topology induced on the set.  Of course, the validity of these properties may depend on the topology   chosen on the space.  A subset of the Euclidean space $\R^N$ is connected if and only if it is path-connected.

\begin{definition}
A subset $E$ of the set of finite measures on a measurable space is called  path-connected in  total variation, if this set is path-connected, when the set of finite measures is endowed with the metric equal to the distance in total variation.
\end{definition}

 A  sequence  $\{\nu_n\}_{n=1,2,\ldots}$ of finite measures on a measurable space $(\Omega,\cal F)$ converges setwise to a measure $\nu$ on $(\Omega,\cal F)$ if for every bounded measurable function $f:\Omega\mapsto\R$
 $
\int_\Omega f(\omega)\nu_n(d\omega)\mapsto\int_\Omega f(\omega)\nu(d\omega).
 $
Setwise convergence defines the topology of setwise convergence of measures; see e.g., Bogachev~\cite[p. 291]{Bo}.
 \begin{definition}
A subset $E$ of the space of finite measures on a measurable space is called  setwise path-connected, if this set is path-connected, when the space of finite measures is endowed with the topology of setwise convergence of measures.
\end{definition}

 In particular, a sequence of occupancy measures $\{Q_n\}_{n=1,2,\ldots}$ converges setwise to an occupancy measure $Q$ if for every bounded measurable function $f:\X\times\A\mapsto\R$
\begin{equation}\label{ewsetcon}\int_\X\int_\A f(x,a)Q_n(dx,da)\to \int_\X\int_\A f(x,a)Q(dx,da).
\end{equation}
In view of \eqref{ewsetcon}, the set  ${\cal M}^\dF$ is setwise path-connected if and only if for every two deterministic policies $\phi^0$ and $\phi^1$   there exists a map $g: [0,1] \mapsto  {\cal M}^\dF$ such that  $g(0)=Q^{\phi^0},$  $g(1)=Q^{\phi^1},$ and the function
  \begin{equation}\label{eqdefzeta}\zeta(\alpha):=\int_\X\int_\A f(x,a)g(\alpha)(dx,da)\end{equation} is continuous
for every bounded measurable function $f:\X\times \A\mapsto\R.$
\begin{theorem}\label{L1} For a uniformly absorbing atomless MDP,
 the set ${\cal M}^\dF$ is path-connected in  total variation and therefore it is setwise path-connected.
\end{theorem}
\begin{proof}  Let $\phi^0$ and $\phi^1$ be two deterministic policies.
Consider  the stationary policy $\pi^*$ defined in \eqref{eqdefpist} and the measure $q$ on $\X$ defined in \eqref{edrfqEFEF}.
The  measure $q$ is atomless in view of Lemma~\ref{lem3.3EF}. So, $q(x)=0$ for all $x\in \X.$

Let $\psi$ be an isomorphic map of $\X$ onto the closed interval $[0,1]$; that is, $\psi$ is a one-to-one measurable mapping of  $(\X,{\cal X})$ onto $([0,1],{\cal B}([0,1])).$  Observe that the function $\psi$ can be viewed as a nonnegative random
variable on the measurable space $(\X,{\cal X})$ with the distribution function
  \[F_\psi(b):= \frac{q(\{x\in \X:~\psi(x)\le b\})}{q(\X)}.\]
In particular, $F_\psi(0)=q(\{\psi^{-1}(0)\})=0,$ and the second equality holds because $ \{\psi^{-1}(0)\}$ is a singleton and the measure $q$ is atomless.  In addition, $F_\psi(1)=1$ because $\{x\in \X:~\psi(x)\le 1\}=\X.$

The distribution function $F_\psi$ is continuous.  Indeed, first observe that $F_\psi(b)=0$ for $b\le 0$ and $F_\psi(b)=1$ for $b\ge 1.$ Second, consider $b\in  [0,1]$ and observe that $F_\psi(b-)=q(\{x\in \X:~\psi(x)< b\}/q(\X),$ $b\in \R.$  Then $F_\psi(b)-F_\psi(b-)=q(\{\psi^{-1}(b)\})=0,$ where the last inequality holds because the set $\{\psi^{-1}(b)\}$ is a singleton and the measure $q$ is atomless.

The continuity of the function $F_\psi$
implies that for $\alpha\in [0,1]$
\[ F^{-1}_\psi (\alpha)=[b_{min}(\alpha),b_{max}(\alpha)],
\]
where
 $b_{min}(\alpha):=\inf\{b\ge 0:~F_\psi(b)=\alpha\};$ $b_{max}(\alpha):=\sup\{b\le 1:~F_\psi(b)=\alpha\},$
and
\begin{equation}\label{eqef7}
q(\psi^{-1}(F_\psi^{-1}(\alpha)))=0.
\end{equation}
We observe that $b_{min}(\alpha)=\inf\{b:~F_\psi(b)\ge\alpha\},$ and this function is well-studied in the literature under the names of the value-at-risk and quantile function.  The function $b_{min}(\alpha)$ is nondecreasing and left-continuous on $[0,1];$ see e.g., Embrechts and Hofert~\cite[Prop. 1(2)]{EH}.  Therefore, it is lower semicontinuous.  Since $F_\psi$ is a continuous function, the function $b_{min}(\alpha)$ is strictly increasing; see e.g., \cite[Prop. 1(7)]{EH}.

Let us consider the collection of increasing subsets $\X_\alpha\subset \X$ and ${\bar \X}_\alpha\subset \X:$
 \begin{equation}\label{edefxalppha}
  \begin{aligned}
\X_\alpha:&=&\{x\in \X:~\psi(x)<b_{min}(\alpha)\}, \qquad\qquad\qquad\qquad\quad \alpha\in [0,1],\\
{\bar \X}_\alpha:&=&\{x\in \X:~\psi(x)\le b_{max}(\alpha)\}=\X_\alpha\cup F^{-1}_\psi(\alpha), \qquad \alpha\in [0,1],
  \end{aligned}
  \end{equation}
and  define the deterministic policies $\varphi_\alpha$ and    ${\bar \varphi}_\alpha:$
\begin{equation}\label{edefxaEEEFEF}
\varphi_\alpha(x):=\left\{
\begin{array}{ll}
\phi^1(x), &{\rm if}\  x\in \X_\alpha,\\
\phi^0(x), &{\rm if}\  x\in \X\setminus \X_\alpha,
\end{array}
\qquad\qquad
\right.
{\bar \varphi}_\alpha(x):=\left\{
\begin{array}{ll}
\phi^1(x), &{\rm if}\  x\in {\bar \X}_\alpha,\\
\phi^0(x), &{\rm if}\  x\in \X\setminus {\bar \X}_\alpha.
\end{array}
\right.
\end{equation}
Observe that $q({\bar \X}_\alpha)=q(\X)F_\psi(b_{\rm max}(\alpha))=q(\X)\alpha,$ as follows from the definition of ${\bar \X}_\alpha.$ According to \eqref{eqef7}, \begin{equation}\label{qqqaEF}
q( \X_\alpha)=q({\bar \X}_\alpha)=q(\X)\alpha.
\end{equation}

Recall that $X(\varphi_\alpha,{\bar \varphi}_\alpha)$ is the set of states on which $\varphi_\alpha$ and ${\bar \varphi}_\alpha$ make the same decisions; see \eqref{enonotcoinEF}.
Since $\X\setminus X(\varphi_\alpha,{\bar \varphi}_\alpha)\subset F^{-1}_\psi(\alpha),$  equality~\eqref{eqef7} and Lemma~\ref{lem6.1} imply that $q^{\varphi_\alpha}=q^{{\bar \varphi}_\alpha}$ for all $\alpha\in [0,1].$ By definition, $\phi^0=\varphi_0$ and $\phi^1={\bar \varphi}_1.$ Thus, $q^{\phi^0}=q^{\varphi_0}$  and $q^{\phi^1}=q^{\varphi_1}.$

Observe that
\begin{equation}\label{edeltaalpha} q(\X\setminus X(\varphi_\alpha, \varphi_{\alpha+\Delta}))=q(\X\setminus X({\bar \varphi}_\alpha, {\bar \varphi}_{\alpha+\Delta}))=  q(\X)|\Delta|,\qquad\alpha,\alpha+\Delta\in [0,1],
\end{equation}
where the last equation holds because \[q(\X\setminus X({\bar \varphi}_\alpha, {\bar \varphi}_{\alpha+\Delta}))=q({\bar \X}_\alpha\bigtriangleup {\bar \X}_{\alpha+\Delta}) = q(\X)|F_\psi(b_{min}(\alpha+\Delta))-F_\psi(b_{min}(\alpha))|,\] where ${\bar \X}_\alpha\bigtriangleup {\bar \X}_{\alpha+\Delta}:=({\bar \X}_\alpha\cup {\bar \X}_{\alpha+\Delta})\setminus({\bar \X}_\alpha\cap {\bar \X}_{\alpha+\Delta})$ is the symmetric difference.
Let us define the  mapping $g,$ 
  \[g(\alpha):= Q^{\varphi_\alpha},\qquad\qquad \alpha\in [0,1].\]
As shown above,  $g(0)=Q^{\phi^0}$ and $g(1)=Q^{\phi^1}.$  Formula \eqref{edeltaalpha} and Therem~\ref{th6.2} imply that this mapping is continuous in  total variation.

\end{proof}
\begin{corollary}\label{cL1}
For a uniformly absorbing atomless MDP  the performance set ${\cal V}^\dF$ is connected.
\end{corollary}
\begin{proof}
Let $\phi^0$ and $\phi^1$ be deterministic policies. Let us consider the function $g:[0,1]\mapsto {\cal M}^{\dF}$ satisfying \eqref{eqdefzeta} for all bounded measurable functions $f.$ The existence of such a function follows from Theorem~\ref{L1}.  Then the vector-function $\tilde{\zeta}(\alpha):=\int_\X\int_\A f(x,a)g(\alpha)(dx,da)$
defines a path connecting $v^{\phi^0}$ and $v^{\phi^1}$ in $\R^N.$
\end{proof}
\begin{corollary}\label{cN1Conv}
If $N=1,$ then  the set ${\cal V}^\dF$ is convex for a uniformly absorbing atomless MDP. 
\end{corollary}
\begin{proof}
Corollary~\ref{cL1} and the mean value theorem imply that the bounded one-dimensional set  ${\cal V}^\dF$ is convex.
\end{proof}
\begin{corollary}\label{cN1}  If $N=1,$ then  ${\cal V}^\dF={\cal V}$ for a uniformly absorbing atomless MDP. \end{corollary}
\begin{proof} Let $v_*:=\inf_{\pi\in \S} v^\pi$ and $v^*:=\sup_{\pi\in \S} v^\pi.$ Then $-\infty<v_*<v^*<+\infty,$ where the first and the last inequality hold since  the MDP is absorbing and the reward function $r$ is bounded.
According to Feinberg~\cite{F92}, $\inf_{\phi\in\dF} v^\phi=  v_*$ and  $\sup_{\phi\in\dF} v^\phi=  v^*.$   These equalities  imply that the closures of the one-dimensional convex sets ${\cal V}^\dF$ and ${\cal V}$ are  both  equal to the closed bounded interval $[v_*,v^*].$    In addition, according to Corollary~\ref{coropt2}, if $v\in \{v_*,v^*\}\cap {\cal V},$ then $v\in  {\cal V}^\dF.$   Therefore,  ${\cal V}^\dF\supset {\cal V}$ and, by definition, ${\cal V}^\dF\subset {\cal V}.$
\end{proof}

\section{Proof of Theorem~\ref{tgmain}}\label{s9}

For the performance set of deterministic policies ${\cal V}^\dF,$ consider its closure ${\bar {\cal V}}^\dF.$  Since the set ${\cal V}^\dF$ is bounded,  ${\bar {\cal V}}^\dF$ is compact.

\begin{lemma}\label{lclos} Under the assumptions of  Theorem~\ref{tgmain}, if the set ${\cal V}^\dF$ is convex, then ${\cal V}\subset {\bar {\cal V}}^\dF.$
\end{lemma}
\begin{proof} Suppose that  $\cal V\not\subset {\bar {\cal V}}^\dF.$  Then there exists a stationary policy $\pi$ such that $v^\pi\notin {\bar {\cal V}}^\dF.$ Therefore, there exists a hyperplane  in $\R^N$ separating the point $v^\pi$ and the convex compact set ${\bar {\cal V}}^\dF.$ Let $\langle b,v\rangle +d=0$ be such a hyperplane, and let  $\langle b,v^\pi\rangle +d>0$ and $\langle b,v\rangle +d\le 0$ for all $v\in {\cal V}^\dF,$ where $b\in\R^N$ and $d\in \R.$ Thus
\begin{equation}\label{eL81EF}
 \sup_{\phi\in\dF}\langle b,v^\phi\rangle=\sup_{v\in\cal{V}^\dF} \langle b,v\rangle <\langle b,v^\pi\rangle.
\end{equation}

Let us consider the reward function ${\tilde r}(x,a):=\langle{b},r(x,a)\rangle,$ where $x\in \X,$ and $a\in A (x).$ The expected total rewards for this reward function, a policy $\sigma,$ and the initial state distribution $\mu$ is denoted by ${\tilde v}^\sigma,$ and    ${\tilde v}^\sigma=\langle b,v^\sigma\rangle$ for all $\sigma\in\S.$

Supremums of the expected total rewards are equal for deterministic and stationary policies; see Feinberg~\cite{F92}.  Therefore, $\sup_{v\in{\cal V}^\dF} \langle b,v\rangle=\sup_{\phi\in\dF}\tilde{v}^\phi \ge\tilde{v}^\pi =  \langle b,v^\pi\rangle.$ This contradicts \eqref{eL81EF}.
\end{proof}

\begin{lemma}\label{lcovdeFF}
Let the statement of Theorem~\ref{tgmain} be correct for $N=1,2,\ldots$ criteria. Then, under the assumptions of Theorem~\ref{tgmain}, the set ${\cal V}^\dF$ is convex  for the case of $(N+1)$ criteria.
\end{lemma}


\begin{proof} 
Let the lemma be correct for $N$-dimensional vector-functions $r,$ where $N=1,2,\ldots$.  We shall prove that the set ${\cal V}^\dF$ is convex for $(N+1)$-dimensional vector-functions $r.$  Let $\phi^0$ and $\phi^1$ be two deterministic policies and $\lambda\in (0,1).$ Our goal is to show that there exists a deterministic policy $\phi_\lambda$ such that
$v^{\phi_\lambda}:=\lambda v^{\phi^0}+ (1-\lambda)v^{\phi^1}.$  Let us consider the stationary policy $\pi^*$ defined in \eqref{eqdefpist}, the measure $q$ on $\X$ defined in
\eqref{edrfqEFEF}, and the family of expanding sets $\X_{\alpha}\subset \X$ defined in \eqref{edefxalppha}.
%
 For each $\alpha\in [0,1]$  we consider the submodel with the action sets reduced to the sets
  $$A^\alpha(x)=\left\{ \begin{array}{ll}
\{\phi^1(x)\}, & \mbox{ if } x\in \X_\alpha, \\
\{\phi^0(x),\phi^1(x) \}, & \mbox{ if } x\in \X\setminus \X_\alpha.
  \end{array} \right.$$
  Let ${\cal V}(\alpha)$ be the set of all performance vectors for the submodel with the action sets $A^\alpha(\cdot).$ According to Lemmas~\ref{lem3.1} and \ref{cfinAabs}, each set ${\cal V}(\alpha)$ is convex and compact.  In addition,
  \begin{equation}\label{eqsetmonEAF}
  {\cal V}(\alpha)\subset {\cal V}(\beta)\qquad {\rm if\ } 0\le \beta\le\alpha\le 1.
  \end{equation}

 In view of the definition in \eqref{edefxalppha}, $\X_0=\emptyset,$ which implies
  \[ A^0(x)=\{\phi^0(x),\phi^1(x)\}, \qquad x\in\X.\]
Therefore ${\cal V}(0)$ is the performance set for the MDP defined by the deterministic policies $\phi^0$ and $\phi^1.$ Thus, $v^{\phi^0}, v^{\phi^1}\in {\cal V}(0).$

Observe that  ${\cal V}{(1)}=\{v^{\phi^1}\}.$
Indeed, let $\varphi$ be a deterministic policy for the MDP with the action sets $A^1(\cdot).$ Then $\varphi(x)=\phi^1$ when $x\in\X_1\subset \X.$ In view of \eqref{qqqaEF}, $q(\X\setminus\X_1)=0.$  Since $\X\setminus X(\varphi,\phi^1)\subset \X\setminus\X_1,$ we have that $q(\X\setminus X(\varphi,\phi^1))=0.$ Lemma~\ref{lem6.1} implies that $q^\varphi=q^{\phi^1}.$  Therefore,
  $v^\varphi=\int_\X r(x,\varphi(x))q^\varphi(dx)=\int_\X r(x,\phi^1(x))q^{\phi^1}(dx)=v^{\phi^1},$
  where the first and the last equalities follow from the definitions of expected total rewards, occupancy measures, and deterministic policies;  the  equality in the middle follows from $q^\varphi=q^{\phi^1}$ and $\varphi(x)={\phi^1}(x)$ for $ q^{\phi^1}$-almost all $x\in\X.$

  Since the set ${\cal V}(0)$ is convex and $v^{\phi^0},v^{\phi^1}\in{\cal V}(0),$ we have that  $\lambda v^{\phi^0}+ (1-\lambda)v^{\phi^1}\in{\cal V}(0).$  Consider an arbitrary point ${\hat v}\in {\cal V}(0).$  We shall prove that $v^\phi={\hat v}$ for some deterministic policy $\phi$ for the submodel with action sets $A^0(x),$ $x\in \X.$

To do this, we'll show that ${\hat v}\in \partial({\cal V}(\hat\alpha))$ for some $\hat\alpha\in [0,1],$ where $\partial(G)$ is the boundary of the convex compact subset $G$ of $\R^{N+1}.$
 For a point $e\in\R^{N+1}$ and a closed set $E\subset\R^{N+1},$ we denote by $d(e,E):=\min\{\Vert e-z\Vert: z\in E\}$ the distance between $e$ and $E.$  Since $E$ is closed, $d(e,E)=0$ if and only if $e\in E.$  If $E_1\subset E_2$ for two closed subsets of $\R^{N+1},$  then $d(e, E_2)\le d(e, E_1).$

 As follows from \eqref{eqsetmonEAF},
 the function \[ G(\alpha):=d({\hat v},{\cal V}(\alpha)),\qquad\qquad \alpha\in [0,1],\] is nondecreasing in $\alpha\in [0,1]$ and  $G(0)=d({\hat v},{\cal V}(0))=0.$  Let us prove that this function is continuous. To do this, we choose an arbitrary $\alpha\in [0,1)$ and $\Delta>0$ such that $\alpha+\Delta\le 1.$ We also choose an arbitrary point $v\in {\cal V}(\alpha).$ Let $\pi$ be a stationary policy in the submodel with the action sets $A^\alpha(x),$ $x\in\X,$ such that $v^{\pi}=v.$ Let $\sigma$ be  the stationary policy in the model with the action sets  $A^{\alpha+\Delta}(x),$ $x\in\X,$ defined by
 \[
\sigma(\phi^1(x)|x):=\begin{cases} 1, &{\rm if}\ x\in \X_{\alpha+\Delta}\setminus \X_{\alpha},\\
\pi((\phi^1(x)|x), &{\rm if}\ x\in \X\setminus (\X_{\alpha+\Delta}\setminus\X_\alpha).
\end{cases}
\]
Then $\X\setminus (\X_{\alpha+\Delta}\setminus\X_\alpha)\subset X(\pi,\sigma),$ which implies $ \X\setminus X(\pi,\sigma)\subset \X_{\alpha+\Delta}\setminus\X_\alpha.$  As follows from \eqref{qqqaEF}, $q(\X\setminus X(\pi,\sigma))\le q(\X)\Delta.$ According to Theorem~\ref{th6.2}, for every $\epsilon_1>0$  there exists $\delta>0$ such that $d_{TV}(Q^{\pi},Q^\sigma)\le \epsilon_1$ if $\Delta\le \delta.$   This implies that $\Vert v^{\pi}-v^\sigma \Vert\le K(N+1)^{\frac{1}{2}}\epsilon_1,$ where the positive constant $K$ is an upper bound of $|r^{(n)}(x,a)|$ for $x\in\X,$ $a\in \A,$ and $n=1,2,\ldots,N+1.$
 So, if we choose an arbitrary $\epsilon>0,$ set $\epsilon_1=\epsilon/(K(N+1)^{\frac{1}{2}}),$ and choose $\Delta \le \delta,$ then $\Vert v^{\pi}-v^\sigma \Vert\le\epsilon.$  This implies that, if $\Delta \le \delta,$  $\alpha\in [0,1),$ and  $\alpha + \delta\le 1,$ then
 \begin{equation}\label{estdegal}
 d(v,{\cal V}(\alpha+\Delta))\le\epsilon \qquad {\rm for\ all\ }v\in{\cal V}(\alpha).
  \end{equation}
Let us consider two cases: (i) $G(\alpha)>0$ and (ii) $G(\alpha)=0.$


 (i) In this case, $\hat{v}\notin {\cal V}(\alpha).$ We denote by $\hat{v}_\alpha$ the projection of the point $\hat{v}$ onto the convex compact set $\cal{V}(\alpha),$  that is,  $\hat{v}_\alpha$ is the unique point in $\cal{V}(\alpha)$ satisfying $\Vert\hat{v}-\hat{v}_\alpha \Vert=d(\hat{v},\cal{V}(\alpha)).$ Let $\hat{v}_{\alpha+\Delta}\in {\cal V}(\alpha+\Delta)$ be the projection of $\hat{v}_\alpha$ onto the compact set ${\cal V}(\alpha+\Delta).$ Then, according to the triangle inequality
\begin{equation*} d(\hat{v},{\cal V}(\alpha))+d(\hat{v}_\alpha,{V}(\alpha+\Delta))=   \Vert \hat{v}-\hat{v}_\alpha\Vert +\Vert \hat{v}_\alpha-\hat{v}_{\alpha+\Delta}\Vert\ge \Vert \hat{v}-\hat{v}_{\alpha+\Delta}\Vert\ge d(\hat{v}, {\cal V}(\alpha+\Delta)).
 \end{equation*}
 Since $0<d(\hat{v}_\alpha,{V}(\alpha+\Delta))<\epsilon$ and the nonnegative function $G(\alpha)$ is nondecreasing, the last formula implies
 \begin{equation}\label{einecas1}
   0\le G(\alpha+\Delta)- G(\alpha)\le \epsilon.
 \end{equation}

(ii) The equality $G(\alpha)=0$ means that $\hat{v}\in {\cal V}(\alpha).$  Therefore, \eqref{estdegal} for $v=\hat v$  implies $0\le G(\alpha+\Delta)- G(\alpha)= G(\alpha+\Delta)\le \epsilon.$  So,  \eqref{einecas1} holds.

Since \eqref{einecas1} holds for the both cases, this implies continuity of the function $G(\alpha)$  on $[0,1].$
Let us define \[{\hat \alpha}:= \max\{\alpha\in [0,1]: d(\hat{v},{\cal V}(\alpha))=0\}.\]
 This point exists because $d(\hat{v},{\cal V}(0))=0$ and the continuous function $G(\alpha)=d(\hat{v},{\cal V}(\alpha))$ is nondecreasing in $\alpha.$ Since $d(\hat{v},{\cal V}(\hat{\alpha}))=0,$ we have that $\hat{v}\in {\cal V}(\hat\alpha).$  If $\hat{\alpha}=1,$ then $\hat{v}=v^{\phi^1}\in {\cal V}(1)=\partial {\cal V}(1)$ since ${\cal V}(1)=\{v^{\phi^1}\}.$ 

 So, we need to consider the case  $\hat{\alpha}\in [0,1).$  In this case we shall prove that $\hat{v}\in\partial({\cal V}(\hat{\alpha})).$

 Since ${\hat v}\in {\cal V}({\hat \alpha}),$ in order to prove that $\hat{v}\in\partial({\cal V}(\hat{\alpha})),$ it is sufficient to show that ${\hat v}$ cannot be an interior point of ${\cal V}({\hat \alpha}).$  Indeed, let ${\hat v}$ be an interior point of  ${\cal V}({\hat \alpha}).$ Then there exists $\epsilon>0$ such that $d({\hat v},\partial ({\cal V}({\hat \alpha})))\ge \epsilon.$  In view of \eqref{estdegal} for $\alpha=\hat\alpha,$ there exists $\Delta>0$ such that ${\hat \alpha}+\Delta\le 1 $ and $d(v,{\cal V}({\hat \alpha}+\Delta))\le \epsilon/2$ for all $v\in {\cal V}({\hat \alpha}).$ Thus, $d(\hat{v},{\cal V}({\hat \alpha}+\Delta))\le \epsilon/2.$ The definition of ${\hat \alpha}$ implies that $d({\hat v}, {\cal V}({\hat \alpha}+\Delta))>0.$  Let $\hat{v}_p$ be the projection of the point ${\hat v}$ onto the convex set ${\cal V}({\hat \alpha}+\Delta).$ Observe that ${\hat v}_p$ is an interior point of  ${\cal V}({\hat \alpha})$ because $\Vert \hat{v}-\hat{v}_p\Vert= d(\hat{v},{\cal V}({\hat \alpha}+\Delta))\le\epsilon/2.$      Since ${\hat v}$ and $\hat{v}_p$ are interior points of ${\cal V}({\hat \alpha}),$ there is a point $v\in \partial({\cal V}({\hat \alpha}))$   such that $v$ belongs to the  line projecting ${\hat v}$ to ${\cal V}({\hat \alpha}+\Delta),$ and ${\hat v}$ is located between ${\hat v}_p$ and $v.$  This is illustrated on Fig.~\ref{Fig01}. Therefore $d(v, {\cal V}({\hat \alpha}+\Delta))= \Vert v-{\hat v}_p \Vert \ge  \Vert v-{\hat v} \Vert \ge d({\hat v},\partial({\cal V}({\hat \alpha})))\ge \epsilon,$ where the first inequality holds because $\hat v$ is between $v$ and $\hat{v}_p,$  the second inequality follows from ${v}\in \partial ({\cal V}({\hat \alpha})),$  and the last one follows from the choice of $\epsilon.$ This conclusion contradicts to $d(v,{\cal V}({\hat \alpha}+\Delta))\le \epsilon/2.$  
 Therefore, $ v\in \partial({\cal V}({\hat \alpha})).$  

 \begin{figure}[!htb]
\centering
	\includegraphics[scale=0.7]{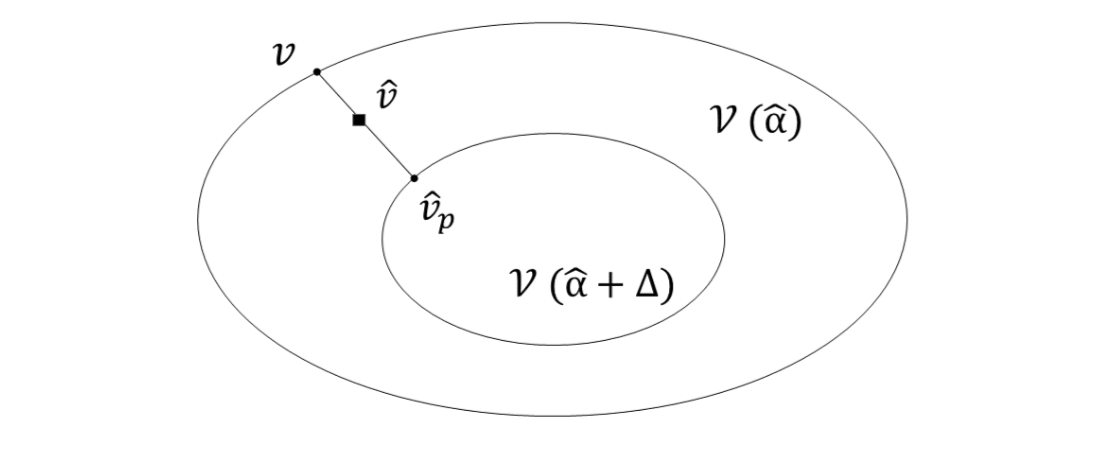}
\caption{\normalsize $\hat v$ cannot be an interior point of ${\cal V}(\hat\alpha)$: otherwise, $d(v, {\cal V}({\hat \alpha}+\Delta))= \Vert v-{\hat v}_p \Vert \ge  \Vert v-{\hat v} \Vert \ge d({\hat v},\partial({\cal V}({\hat \alpha})))\ge \epsilon,$	 and $d(v,{\cal V}({\hat \alpha}+\Delta))\le \epsilon/2$ (contradiction).	 \label{Fig01}}
\end{figure}

 Since ${\hat v}\in \partial({\cal V}({\hat \alpha})),$ by Theorem~\ref{tdimred} there is a coordinate $i=1,\ldots,N+1$ such that ${\hat v}_{-i}$ is a performance vector in a submodel of the MDP with action sets $A^{\hat \alpha}(\cdot)$ and the value of $v^{(i)}$ is completely defined by the vector  ${\hat v}_{-i}$ according to formula \eqref{edimred}. The vector ${\hat v}_{-i}$ has $N$ coordinates. By the induction assumption, there is a deterministic policy $\phi$ such that   $ v^{\phi}_{-i}={\hat v}_{-i}.$  Thus,   $ v^{\phi}={\hat v}.$
\end{proof}

\begin{proof}[Proof of Theorem~\ref{tgmain}]
According to Corollary~\ref{cN1}, the statement of the theorem is correct for $N=1.$ Suppose the statement of Theorem~\ref{tgmain} is correct for $N$ criteria, where $N=1,2,\ldots\ .$ Let us prove that it is correct for the case of $(N+1)$ criteria.

Consider the case on $(N+1)$ criteria. By Lemma~\ref{lcovdeFF}, the set ${\cal V}^\dF$ is convex. Therefore,
Lemma~\ref{lclos}  and ${\cal V}^\dF\subset \cal V$ imply that, if $v\in {\cal V}\setminus \partial({\cal V}),$ then $v\in {\cal V}^\dF.$  Let $v\in \partial({\cal V}).$ Theorem~\ref{tdimred} implies that there exists a coordinate $i=1,\ldots,N+1,$  a vector $b\in\R^N, $ a constant $d,$ and a submodel with the performance set $\tilde{\cal{V}}$
such that $ v\in{\tilde {\cal V}}$ and     ${\tilde v}^{(i)}= d+\langle b,{\tilde v}_{-i}\rangle$ for all ${\tilde v}\in{\tilde {\cal V}},$  where for $w\in\R^N$ the following notations are used: $w^{(i)}$ is the $i^{\rm th}$ coordinate of the vector $w$ and $w_{-i}$ is the projection of $w$ onto $\R^N$ obtained by removing the $i^{\rm th}$ coordinate from $w.$ As follows from the induction assumption, there is a deterministic policy $\phi$ in the submodel such that $v^\phi_{-i}=v_{-i}$ and $v^{(i),\phi}=  d+\langle b, v^\phi_{-i}\rangle  = d+\langle b,v_{-i}\rangle =v^{(i)}.$ Thus, $v^\phi=v.$
\end{proof}

\section{Unbounded Rewards}\label{s10}
This section describes extensions to unbounded reward vector-functions $r.$  These extensions are based on the standard weighted norm transformation of an MDP with unbounded rewards to an MDP with bounded rewards.  
%
%

Let us consider an MDP with the expected total rewards and with a standard Borel state space $\bar{\X}:=\X\cup\{\bar{x}\}, $ where ${\bar x}\notin \X,$ standard Borel action space $\A,$ sets of available actions $A(x),$ where $A({\bar x})=\{{\bar a}\},$ with $\bar a$ being an arbitrary point in $\A,$ transition probabilities $p$ such that $p({\bar x}|{\bar x},{\bar a})=1,$ a reward vector-functions $r$ with values in $\R^N$  such that $r^{(n)}({\bar x},{\bar a})=0,$ $n=1,2,\ldots,N,$ and  an initial probability distribution $\mu$ such that $\mu(\X)=1.$ Let there exist a positive measurable function $w:\X\mapsto (0,+\infty),$ for which the following conditions hold:

(a)  $\sup_{x\in \X}\sup_{a\in A(x)}\frac{1}{w(x)}\int_\X w(y)p(dy|x,a)\le 1,$

(b) $\int_\X w(x)\mu(dx)< +\infty,$

(c) $\sup_{x\in \X}\sup_{a\in A(x)}\frac{|r^{(n)}(x,a)|}{w(x)}< +\infty,$ $n=1,2,\ldots,N.$


Let us consider an MDP with state space $\bar \X,$ action space $ \A,$ sets of available action $A(x),$  $x\in\bar \X,$ transition probability $\tilde p,$ where ${\tilde p}(\bar{x}|\bar{x},\bar{a}):=1,$  \[{\tilde p}(Y|x,a):=\frac{1}{w(x)}\int_Yw(y)p(dy|x,a),\qquad \qquad  Y\in {\cal X},\  x\in \X,\  a\in A(x),\]
and
\[ {\tilde p} ({\bar x}|x,a):=1-\frac{1}{w(x)}\int_\X w(y)p(dy|x,a),\qquad \qquad    x\in \X,\  a\in A(x),\]
reward function $\tilde r,$ where ${\tilde r}^{(n)}(\bar{x},\bar{a})=0$ and,  for $n=1,2,\ldots,N,$
\begin{equation*}{\tilde r}^{(n)}(x,a)=\frac{r^{(n)}(x,a)}{w(x)}\int_\X w(y)\mu(dy),\qquad\qquad x\in \X,\ a\in A(x),\end{equation*}
  and the initial probability distribution $\tilde\mu$ with
  \begin{equation}\label{edefwtkl}\tilde{\mu}(Y):=\frac{\int_Y w(x)\mu(dx)}{\int_\X w(y)\mu(dy)},\qquad\qquad Y\in\cal \cal X,\end{equation}
   and  $\mu(\bar{x})=0.$ If $\mu(x)=0,$ then $\tilde\mu(x)=0,$ $x\in\X.$  Let ${\tilde v}^\pi$ be the vector of the expected  total expected rewards in the MDP with the transition probabilities $\tilde p$ and rewards $\tilde r$ controlled by a policy $\pi,$ when the initial state distribution is $\tilde{\mu}.$

We say that the defined MDP is 
 uniformly absorbing, if  equality \eqref{elimsup} holds for this MDP with the initial distribution $\tilde\mu$ instead of $\mu$  and the transition probability $\tilde p$ instead of $p.$
This definition is consistent with Definition~\ref{defabsu} because the assumptions in Definition~\ref{defabs} also hold for this MDP with the fixed initial state distribution $\tilde\mu.$ In addition,    the function $\tilde r$ is bounded.   The following statement follows from Theorem~\ref{tgmain}.
\begin{corollary}\label{cwn1}
Consider an MDP with the state space $\bar \X$ satisfying conditions (a--c) and such that $r({\bar x},{\bar a})=0$ and $p({\bar x}|{\bar x},{\bar a})=1$ for the state ${\bar x}$ and action ${\bar a}$ defined above. Then ${\tilde v}^\pi=v^\pi$ for all policies $\pi.$  Furthermore, if the MDP with the transition probabilities $\tilde p$ is uniformly absorbing and atomless, then ${\cal V}={\cal V}^\dF$ for the initial MDP and this set is convex.
\end{corollary}
\begin{proof}
Let $\tilde E$ and $\tilde P$ denote the expectations and probabilities for the MDP with the transition probabilities $\tilde p$ and the initial distribution $\tilde\mu.$
$P^\pi(dx_0da_0\ldots dx_tda_t)$ and ${\tilde P}_{\tilde \mu}^\pi(dx_0da_0\ldots dx_tda_t)$ are probability distributions on the standard Borel space $({\bar \X}\times\A)^{t+1},$ where $t=0,1,\ldots\ .$  The standard straightforward arguments imply that  for $(x_0,a_0,\ldots,x_t,a_t)\in (\X\times\A)^{t+1},$ $t=0,1,\ldots,$
\begin{equation}\label{eqwn01}{\tilde P}_{\tilde\mu}^\pi(dx_0da_0,\ldots,x_ta_t) =\frac{w(x_t)P^\pi(dx_0da_0,\ldots,x_ta_t)}{\int_\X w(y)\mu(dy)}.\end{equation}
%

 Since $p({\bar x}|{\bar x}, {\bar a})={\tilde p}({\bar x}|{\bar x}, {\bar a})=1$ and $r({\bar x},{\bar a})={\tilde r}({\bar x},{\bar a})=0,$   equality~\eqref{eqwn01} and the definition of the reward function ${\tilde r}$ imply that $E^\pi r(x_t,a_t)={\tilde E}_{\tilde \mu}^\pi {\tilde r}(x_t,a_t)$ for all $t=0,1,\ldots\ .$  This equality implies that
${\tilde v}^\pi=v^\pi$ for an arbitrary policy $\pi.$  This implies that ${\cal V}^\G=\{{\tilde v}^\pi:\pi\in \G\}$ for every set of policies $\G\subset\Pi.$   Since the MDP with the transition probabilities $\tilde p$ is uniformly absorbing and the reward vector-function $\tilde r$ is bounded,  Theorem~\ref{tgmain} implies that $\{{\tilde v}^\phi:\phi\in\dF\}=\{{\tilde v}^\pi:\pi\in\Pi\}.$  Therefore, ${\cal V}={\cal V}^\dF.$
\end{proof}

Now let us consider a discounted MDP with the state space $\X$ introduced in Section~\ref{s2} without assuming that the reward vector-function $r$  is bounded.  Let us consider the following assumption:

(d) there exists a positive measurable function $w:\X\mapsto (0,+\infty)$ satisfying assumptions  (b,c), and
 there exists a constant $\tilde\beta\in (0,1)$ such that $\beta\sup_{x\in \X}\sup_{a\in A(x)}\frac{1}{w(x)}\int_\X w(y)p(dy|x,a)\le \tilde\beta.$

Then the following corollary from Theorem~\ref{tmain} holds.
\begin{corollary}\label{cwn2}
If an atomless discounted MDP with a possibly unbounded reward vector-function $r$ satisfies assumption (d), then  ${\cal V}_\beta={\cal V}_\beta^\dF$ and this set is convex.
\end{corollary}
\begin{proof}
Let us add an isolated point $\bar x$ to the standard Borel space $\X$ and set ${\bar \X}:= \X\cup\{\bar x\}.$ Let us consider a discounted MDP with the action set $\A,$ sets of available actions $A(x),$ $x\in \X,$ reward vector-function $\tilde r,$ initial state distribution $\tilde \mu$ and the discount factor $\tilde\beta$ described and defined above.  However, instead of $\tilde p,$ the transition probability for this MDP is $\hat p,$ where ${\hat p}(\bar{x}|\bar{x},\bar{a}):=1,$  \[ {\hat p}(Y|x,a):=\frac{\beta}{{\tilde\beta}w(x)}\int_Yw(y)p(dy|x,a),\qquad \qquad  Y\in {\cal X},\  x\in \X,\  a\in A(x),\]
and
\[  {\hat p} ({\bar x}|x,a):=1-\frac{\beta }{{\tilde\beta}w(x)}\int_\X w(y)p(dy|x,a),\qquad \qquad    x\in \X,\  a\in A(x).\]

Let $\hat E$ and $\hat P$ denote the expectations and probabilities for the defined MDP with the state space $\bar \X$ and  transition probabilities $\hat p.$ In particular, ${\tilde P}_{\tilde \mu}^\pi(dx_0da_0\ldots dx_tda_t)$ is a probability distribution on the standard Borel space $({\bar \X}\times\A)^{t+1},$ where $t=0,1,\ldots\ .$   The following formula is similar to \eqref{eqwn01}: for $t=0,1,\ldots$ and $(x_0,a_0,\ldots,x_t,a_t)\in (\X\times\A)^{t+1},$
\begin{equation}\label{eqwn02}{\tilde\beta}^t{\hat P}_{\tilde\mu}^\pi(dx_0da_0,\ldots,dx_tda_t) =\frac{\beta^tw(x_t)P^\pi(dx_0da_0,\ldots,dx_tda_t)}{\int_\X w(y)\mu(dy)}.\end{equation}

 Since ${\hat p}({\bar x}|{\bar x}, {\bar a})=1$ and ${\tilde r}({\bar x},{\bar a})=0,$   equality~\eqref{eqwn02} and the definition of the reward function ${\tilde r}$ imply that $\beta^tE^\pi r(x_t,a_t)={\tilde\beta}^t{\tilde E}_{\tilde \mu}^\pi {\tilde r}(x_t,a_t)$ for all $t=0,1,\ldots\ .$  This equality implies that
${\tilde v}_{\tilde\beta}^\pi=v_\beta^\pi$ for an arbitrary policy $\pi,$ where ${\tilde v}_{\tilde\beta}^\pi$ is the vector of the total discounted expected rewards in the MDP with the transition probabilities $\hat p$ and discount factor $\tilde\beta,$ when a policy $\pi$ is chosen and the initial state distribution is $\tilde{\mu}.$  This implies that ${\cal V}^\G_\beta=\{{\tilde v}^\pi_{\tilde{\beta}}:\pi\in \G\}$ for every set of policies $\G\subset\Pi.$   Since the reward vector-function $\tilde r$ is bounded,  Theorem~\ref{tmain} implies that $\{{\tilde v}_{\tilde\beta}^\phi:\phi\in\dF\}=\{{\tilde v}_{\tilde\beta}^\pi:\pi\in\Pi\}.$  Therefore, ${\cal V}_\beta={\cal V}_\beta^\dF.$
\end{proof}
Corollary~\ref{cwn2} can be also proved by reducing discounted MDPs with discounted factors $\beta$ and $\tilde\beta$ to undiscounted MDPs, as this is done in the proof of Lemma~\ref{lreduction}, and by applying Corollary~\ref{cwn1}.

\section{Compactness of Performance Sets and Lyapunov's Convexity Theorem}\label{s11}

In this section we describe sufficient conditions for the compactness of the sets $\cal V$ and ${\cal V}^\dF$ and discuss the relation of our results to Lyapunov's convexity theorem.  From an intuitive point of view, it is clear that the set of the ranges of  vector-measures is a particular case of the sets $\cal V$ and ${\cal V}^\dF$, when a one-step problem is considered.  We demonstrate this in Example~\ref{ex10.2}.
The following example shows  that the set $\cal V$ may be noncompact.
\begin{example}\label{exnoncomp}
Let $\X:=[0,1],$ $A(x):=\A:=(0,1),$ $r(x,a)= a,$ $\mu$ be a Lebesgue measure on $[0,1],$ and under every decision the process moves from every state $x\in\X$ to an absorbing state.  For every deterministic policy $\phi$, we have that $v^\phi=\int_0^1 \phi(x)dx,$ where $\phi:[0,1]\mapsto (0,1)$ is an arbitrary Borel function.  In this example, ${\cal V}^\dF=(0,1).$ Since this MDP is uniformly absorbing and atomless, ${\cal V}={\cal V}^\dF=(0,1).$ By changing the action sets to $(0,1],$ $[0,1),$ and $[0,1],$ we obtain MDPs with performance sets $(0,1],$ $[0,1),$ and $[0,1]$ respectively.
\end{example}

As stated in Corollary~\ref{cafin}, Condition (S) from Section~\ref{s5} is sufficient for the compactness of $\cal V.$ For example, in Example~\ref{exnoncomp} this condition holds when $A(x)=\A=[0,1],$ $x\in\X.$   Condition (S) always holds when all the action sets $A(x)$ are finite. Another sufficient condition (W)  for the compactness  of the set of strategic measures was introduced by Sch\"al~\cite{b8}. This condition assumes  weak continuity of transition probabilities. Being combined with continuity of the bounded reward vector-functions $r:\X\times\A\mapsto \R^N,$ this weak continuity condition implies compactness of the performance set $\cal V.$  This weak continuity condition (W) was used in Feinberg and Piunovskiy~\cite{b6}. We do not use and do not consider weak continuity condition (W) in this paper.  In general, a measure $\nu$ is called atomless if for any measurable set $E$ with $\nu(E)>0$ there exists a measurable subset $E'$ of $E$ such that $\nu(E)>\nu(E')>0.$ A vector-measure is called atomless, if each of its coordinates is an atomless measure.

 Lyapunov's convexity theorem states that the range of a finite atomless  vector-measure is convex and compact. In other words, if $(\X,\cal X)$ is a measurable space and $\nu$ is a finite atomless vector-measure with values in $\R^N,$ then the set ${\cal W}:=\{\nu(B):B\in\cal X \}$ is a compact and convex subset of $\R^N.$

One of the  equivalent formulations of this version of Lyapunov's convexity theorem (see e.g., Blackwell~\cite{Bl}) states that, if $\mu$ is a finite atomless measure on a measurable space $(\X,\cal X)$ and $r:(\X,{\cal X}) \mapsto (\R^N,\B(\R^N))$ is a measurable vector-function, whose coordinates are nonnegative functions satisfying $\int_\X r^{(n)}(x)\mu(dx)<+\infty,$ where $n=1,\ldots,N,$ then the set ${\cal W}^*:=\{\int_B r(x)\mu(dx):B\in\cal \cal{X} \}$ is a compact and convex subset of $\R^N.$

To see that the classic Lyapunov convexity theorem is  equivalent to this statement, for an atomless vector-measure $\nu=(\nu^{(1)},\ldots,\nu^{(N)}),$  define the atomless measure $\mu=\sum_{n=1}^N\nu^{(n)}.$ Since $\nu^{(n)}\ll\mu,$ there are Radon-Nikodym derivatives $r^{(n)}:=d\nu^{(n)}/d\mu,$ $n=1,\ldots,N.$  Therefore, $\nu(B)=\int_B r(x) \mu(dx)$ for all $\in\cal X,$ and ${\cal W}={\cal W}^*.$  Conversely, for an atomless finite measure $\mu$ and  the vector function $r$ described in the previous paragraph,
 $\nu(B)=\int_B r(x)\mu(dx),$ where $B\in\cal X,$ is the atomless  vector-measure, and $\cal{W}^*={\cal W}$ is its range.


The following example demonstrates that Theorem~\ref{tgmain} and Corollaries~\ref{cfinAabs}, \ref{cwn1} imply Lyapunov's convexity theorem for the case, when an atomless measure is defined on a standard Borel space.

\begin{example}\label{ex10.2}
Let us consider an MDP with a state  space $\bar{\X}=\X\cup\{\bar{x}\},$ where $\X$ is a standard Borel space, action sets $A(x):=\A:=\{0,1\}$ and $A(\bar x)=\{0\},$ rewards $r:\X\times\A\mapsto\R^N,$  and $\mu$ being an atomless initial probability measure on $\X.$ We also set $p(\bar x|x,a)=1$ for all $x\in\bar X$ and $a\in A(x).$ That is, from each state $x$  the process moves to the absorbing state $\bar x.$   
We also set $r(x,0):={\bar 0}$ for all $x\in\bar X,$ where $\bar 0$ is the zero-vector in $\R^N,$ and   $r(x,1):=r(x),$ $x\in X,$ where $r=(r^{(1)},\ldots,r^{(N)})$ is a Borel vector-function such that each coordinate function $r^{(n)}$ is nonnegative and $\int_\X r^{(n)}(x)\mu(dx)<+\infty$ for all $n=1,\ldots,N.$ 

Every deterministic policy $\phi\in\dF$ is defined by the set $B^\phi:=\{x\in\X:\phi(x)=1\}.$ Observe that $v^\phi=\int_{B^\phi} r(x)\mu(dx).$  In addition, $\{B^\phi:\, \phi\in\dF\}$ is the Borel $\sigma$-algebra on $\X.$  Thus,  we are in the framework of the equivalent formulation of Lyapunov's convexity theorem,   and ${\cal W}^*={\cal V}^\dF.$ Since the function $r$ can be unbounded, we define the weight function
$w(x):=1+\sum_{n=1}^N |r^{(n)}(x)|,$ $x\in\X.$

Then $v^\phi=\int_{B^\phi} {\tilde r}(x){\tilde \mu}(dx),$ where the measure $\tilde\mu$ is defined in \eqref{edefwtkl} and the vector-function
${\tilde r}(x):=r(x)(w(x))^{-1}\int_\X w(y)\mu(dy),$  $x\in \X,$
is bounded.  Therefore, in view of Corollary~\ref{cwn1},  ${\cal V}^\dF= {\cal V}$ and this set is closed and compact. The compactness of the set $\cal V$ follows from Corollaries~\ref{cfinAabs} and \ref{cwn1}.  The set ${\cal W}^*={\cal V}^\dF$ is convex and compact. Thus, Lyapunov's convexity theorem for a standard Borel space $\X$  is a  particular example of an application of Corollary~\ref{cwn1}, which in its turn follows from Theorem~\ref{tgmain}.
\end{example}

%
%
%

\end{document}